\documentclass{lms}
\usepackage{geometry}                
\usepackage[parfill]{parskip}    
\usepackage{graphicx}
\usepackage{amssymb,amsmath,latexsym}
\usepackage{tikz}
\usepackage{tikz-cd}
\usetikzlibrary{arrows,decorations.markings,decorations.pathreplacing,calc,matrix,intersections}
\tikzset{->-/.style={decoration={markings,mark=at position #1 with {\arrow{>}}},postaction={decorate}}}
\usepackage{dsfont}
\usepackage{mathrsfs}

\usepackage{color}
\usepackage[normalem]{ulem}

\newcommand\ndot{{\mkern 1mu\cdot\mkern 1mu}}

\newcommand{\Z}{{\mathbb{Z}}}
\newcommand{\R}{{\mathbb{R}}}

\newcommand{\RP}{{\mathbb{P}}} 
\newcommand{\vx}{{\bf x}} 
\newcommand{\J}{P} 
\newcommand{\Prod}{\prod\RP^A_{\geq 0}}

\newcommand{\cA}{{\mathcal{A}}}
\newcommand{\cB}{{\mathcal{B}}}

\newcommand{\cI}{{\mathcal{I}}}

\newcommand{\cT}{{\mathcal{F}}}
\newcommand{\cW}{{\mathcal{W}}}
\newcommand{\cZ}{{\mathcal{Z}}}

\newcommand{\bZ}{{\pmb Z}}
\newcommand{\bI}{{\pmb I}}
\newcommand{\bR}{{\pmb{R}}}
\newcommand{\bS}{{\pmb{S}}}
\newcommand{\bT}{{\pmb{F}}}
\newcommand{\im}{\operatorname{im}}
\newcommand{\lk}{\operatorname{lk}}
\newcommand{\rank}{\operatorname{rank}}

\newcommand{\X}{\overline X}

\newcommand{\einv}{\overline e}
\newcommand{\finv}{\overline f}

\newcommand{\bdry}{\partial}
\newcommand{\Out}{{\mathrm{Out}}}  

\newcommand{\Outn}{\Out(F_n)}

\newcommand{\incl}{\hookrightarrow}
\newcommand{\iso}{\cong}

\newcommand{\On} {\mathcal{O}_n}  
\newcommand{\BFn}{b\On}    
\newcommand{\BVn}{\mathcal J_n}   

\newcommand{\core}{\operatorname{core}} 
\newcommand{\csigma}{\sigma} 
\newcommand{\cOsigma}{\bar\sigma} 
\newcommand{\osigma}{\mathring\sigma} 
\newcommand{\scrC}{\mathscr C} 
\newcommand{\sslash}{/\mkern-6mu/}

\definecolor{red}{rgb}{1,0,0} 

\definecolor{darkgreen}{rgb}{0, .5, 0.2}

\definecolor{purple}{rgb}{.7, 0, 1}

\newcommand{\hgline}[2]{.  
\pgfmathsetmacro{\thetaone}{#1}
\pgfmathsetmacro{\thetatwo}{#2}
\pgfmathsetmacro{\theta}{(\thetaone+\thetatwo)/2}
\pgfmathsetmacro{\phi}{abs(\thetaone-\thetatwo)/2}
\pgfmathsetmacro{\close}{less(abs(\phi-90),0.0001)}
\ifdim \close pt = 1pt
    \draw[blue] (\theta+180:1) -- (\theta:1);
\else
    \pgfmathsetmacro{\R}{tan(\phi)}
    \pgfmathsetmacro{\distance}{sqrt(1+\R^2)}
    \draw[blue] (\theta:\distance) circle (\R);
\fi
}

\newtheorem{proposition}{Proposition}[section]
\newtheorem{definition}[proposition]{Definition}
\newtheorem{theorem}[proposition]{Theorem}
\newtheorem{lemma}[proposition]{Lemma}
\newtheorem{corollary}[proposition]{Corollary}

\newnumbered{notation}{Notation}
\newnumbered{conjecture}[proposition]{Conjecture}
\newnumbered{question}[proposition]{Question}
\newnumbered{example}[proposition]{Example}
\newnumbered{remark}[proposition]{Remark}
\newnumbered{comment}{Comment}
\author{Kai-Uwe Bux, Peter Smillie and Karen Vogtmann}

\title[On the bordification of Outer space] {On the bordification of Outer space}

\classno{20E36 (primary),  20F65, 20E05, 57M07 (secondary).}

\begin{document}

\maketitle

\begin{abstract}
We give a simple construction of an equivariant deformation retract of Outer space which is homeomorphic to the Bestvina-Feighn bordification.  This results in a much easier proof that the bordification is (2n-5)-connected at infinity, and hence that $Out(F_n)$ is a virtual duality group.
 \end{abstract}

\section*{Introduction}
The action of $SL(n,\Z)$ on the symmetric space $X_n=SL(n,\R)/SO(n)$ is not cocompact, but Siegel \cite{S} showed how to glue affine spaces to $X_n$ to obtain a contractible manifold with corners in such a way that the action extends to a proper cocompact action.  In a landmark paper  \cite{BS}, Borel and Serre generalized this to all arithmetic groups $\Gamma$ in reductive algebraic groups $G$, and it is now commonly referred to as the {\em Borel-Serre bordification} of the symmetric space $G/K$. In the case $G=SL_n$, Grayson  \cite{G} later showed  how to construct an equivariant deformation retract of $X_n$ with the same properties as the bordification.  Leuzinger  \cite{Leu, Leu2} then defined a similar retract for more general groups, and said it is ``likely that this retract is isomorphic  to the Borel-Serre bordification."   These retracts
avoid many of the technical problems associated with extending the space and the action, and are generally much easier to understand.    

The group $\Outn$ shares a large number of properties with arithmetic groups, many of which are proved by considering its action on {\em Outer space} $\On$, which serves as a substitute for the homogeneous space $G/K$.  Motivated by the work of Borel and Serre, Bestvina and Feighn  \cite{BF} constructed a contractible bordification $\BFn$ such that  the action of $Out(F_n)$ on $\On$ extends to a proper cocompact action on $\BFn$.  They used their bordification  to prove that $Out(F_n)$ is a virtual duality group;  the key further ingredient needed for this is to prove that $\BFn$ is $(2n-5)$-connected at infinity.  

In this paper we follow the lead of Grayson and Leuzinger by showing that there is an equivariant deformation retract of $\On$ which is cocompact and $(2n-5)$-connected at infinity.  We show that this retract is equivariantly homeomorphic to the  Bestvina-Feighn bordification and in the process answer a question in their paper about the topology of the pieces $\Sigma(G,g)$ from which their bordification is constructed.    The description of the retract is simpler than that of the bordification. We note that our retract is equivalent to a retract sketched briefly without proof, discussed mainly for $n=3$, in \cite{Ji}.  We use our description to give a different, considerably simpler  proof of the connectivity result in \cite{BF}. In a sequel we will also use it to study the boundary.

{\bf Acknowledgements:}
Karen Vogtmann was partially supported by  the Humboldt Foundation and a Royal Society Wolfson Award.  Kai-Uwe~Bux was supported by the German Science Foundation via the {\small CRC}\,701.

\section{Background: Outer space}\label{sec:background}

In this section we briefly describe Outer space and its decomposition into open simplices $\osigma(G,g),$ in order to introduce the notation needed for this paper.  For a somewhat more detailed quick introduction to these ideas, see \cite{V}, and for detailed proofs see the original paper \cite{CV}.   

Outer space $\On$ is a contractible, $(3n-4)$-dimensional space with a proper action of the group $Out(F_n)$ of outer automorphisms of the free group $F_n$.  A point of $\On$ is determined by a metric graph $G$ together with a homotopy equivalence $g,$ called a {\em marking}, from a fixed $n$-petaled rose $R_n$ to $G$. The graphs $G$ must be connected with no univalent or bivalent vertices, and we will also assume they have no separating edges; this is sometimes called {\em reduced} Outer space.  The metric on $G$ must have {\em volume 1}, i.e. the sum of the edge lengths is equal to 1. The pair $(G,g)$ is called a {\em marked graph}.  Different marked graphs determine the same point of $\On$ if they are isometric by  an isometry  which commutes with the marking up to homotopy.

There is a natural decomposition of $\On$ as a disjoint union of open simplices.  The simplex  containing the point $(G,g)$ is obtained by simply varying  the (positive) edge lengths of $G$ while keeping the  volume equal to 1; it is denoted $\osigma(G,g)$ and its closure in $\On$ is denoted $\cOsigma(G,g)$. If $G$ has $k$ edges, then $\osigma(G,g)$ is an open $(k-1)$-simplex.  A simplex $\osigma(G',g')$ is a face of $\cOsigma(G,g)$ if $G'$ can be obtained from $G$ by shrinking some edges to points, and $g'$ is homotopic to $g$ composed with the collapse.   Note $\cOsigma(G,g)$  is not a closed simplex since some of its faces  are missing, namely those approached by shrinking all edges of a subgraph that contains loops.

The open simplices $\osigma(G,g)$ form a partially ordered set  ({\em poset\,}) where the partial order is the face relation, i.e. $\osigma(G',g')\leq \osigma(G,g)$ if $\osigma(G',g')\subseteq \cOsigma(G,g).$ The geometric realization of this poset is called the {\em spine} $K_n$ of $\On$.  The spine   has a natural embedding into $\On$ as an equivariant deformation retract (the case $n=2$ is illustrated in  Figure~\ref{OS2}).

A key notion in the paper \cite{BF} is that of a {\em core subgraph} of $G$. By a {\em subgraph} we mean the closure in $G$ of a set of edges.  An (open) edge $e$ of a subgraph $H$ {\em separates} $H$ if $H-e$ has an additional component, and a subgraph $H$ is {\em core} if none of its edges separates $H$.
 In particular, $G$ is a core subgraph of itself. In general core subgraphs need not be connected, and they may contain bivalent vertices (but not univalent vertices, because removing the unique edge to a univalent vertex would separate the subgraph).   Note that every subgraph of $G$ contains a unique maximal core subgraph. If this maximal core  is empty the subgraph is a union of trees, i.e. a {\em forest} in $G$.

\section{Jewels and the retract}\label{jewels}\label{sec:jewels}

In this section we find a compact cell $J(G,g)$ (a {\em jewel\,}) inside the closure $\cOsigma(G,g)$ of $\osigma(G,g)$ in $\On$, then glue these cells together to form an equivariant deformation retract of $\On$.  The construction of $J(G,g)$ will be independent of the marking $g$, so we temporarily eliminate the marking from our notation.

We consider $\osigma(G)$ to be the interior of  a closed simplex $\csigma(G)$.  The faces of $\sigma(G)$ which are not  in $\cOsigma(G)$  are said to be {\em at infinity}.
 We view $\csigma(G)$  as a regular Euclidean simplex; the  lengths of edges in $G$ give barycentric coordinates on $\csigma(G)$. A face at infinity is obtained by setting the edge-lengths of   $H$ equal to zero for some subgraph $H$ that contains a loop.  In particular all vertices of $\csigma(G)$ are at infinity, assuming $n>1$.  An equivalent way of saying this is that  $\rank(H_1(G\sslash H^c))<\rank(H_1(G))$,  where  $G\sslash H^c$ denotes the graph obtained by collapsing all edges  of $H^c$ to points.

The jewel $J(G)$ is a convex polytope, obtained by shaving off some  faces of $\csigma(G)$ that are  at infinity.   Specifically, label each vertex of $\csigma(G)$  by the corresponding edge of $G$.  If  a set of edges  forms a core subgraph of $G$, we shave off the opposite face, i.e. the face spanned by the remaining edges. We shave deeper for larger core graphs.

More precisely: if $G$ has  edges $e_0,\ldots,e_m$, realize $\csigma(G)$ as the set of points in the positive orthant of $\R^{m+1}$ whose coordinate sum is equal to $1$:
$$\csigma(G)=\{(x_0,\ldots,x_{m})\in \R^{m+1} \mid   x_i\geq 0, \, x_0+\ldots+x_m=1\}.$$  
If $A$ is a subgraph of $G$, then there is a natural inclusion $\sigma(A) \subset \sigma(G)$.   For each core subgraph $A$  we shave  the opposite  face $\csigma(A^c)$ by a factor  $c_{A};$ this is accomplished by taking the intersection of   $\csigma(G)$ with the half-space $\sum_{e_i\in A} x_i\geq c_{A}$.
Here the constants $c_A$ should be chosen to be small positive numbers which increase quickly with the size of $A$.  Specifically, we require all $c_{A}\ll 1$ and  $c_{A}>2c_{B}$ if $A$ properly contains $B$. The cell $J(G)$ is the result of shaving all faces opposite core faces of $\csigma(G)$.

 \begin{remark} The jewel $J(G)$ has also found applications in the context of Feynman integrals, where core subgraphs are called {\em 1-particle irreducible subgraphs} and $J(G)$ is called the {\em graph polytope} (see, e.g., \cite{Bro}). In this context the constants $c_A$ depend only on the number of edges in $A$.  In our context it will be more convenient to use constants $c_A$ that depend on the rank of $H_1(A)$.
\end{remark}  

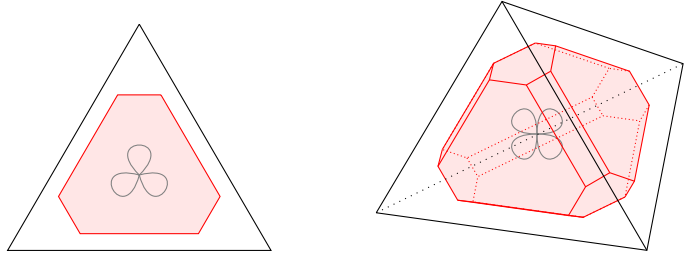
\begin{figure}
\begin{center}
\begin{tikzpicture}[scale=1] 
\draw (90:2) to (210:2) to (330:2) to (90:2);
\fill [red!10] (75:1.1) to (105:1.1) to (195:1.1) to (225:1.1) to (315:1.1) to (345:1.1) to (75:1.1);
\draw [red] (75:1.1) to (105:1.1) to (195:1.1) to (225:1.1) to (315:1.1) to (345:1.1) to (75:1.1);
\draw  [black!50](0,0) .. controls (45:.75) and (135:.75)   .. (0,0);
\draw  [black!50](0,0) .. controls (15:.75) and (285:.75)   .. (0,0);
\draw  [black!50](0,0) .. controls (255:.75) and (165:.75)   .. (0,0);
 \end{tikzpicture}
 \hskip .5in
  \begin{tikzpicture}[scale=1] 
\coordinate (a) at (90:2); 
\coordinate (b) at (330:2); 
\coordinate (c) at (200:1.6); 
\coordinate (d) at ($1.25*(a) + 1.25*(b)$); 
\coordinate (temp) at ($(a)+(b)+(c)+(d)$);
\coordinate (center) at ($.25*(temp)$);
\coordinate (aplus) at ($1.1*(a) - .1*(center)$);
\coordinate (bplus) at ($1.1*(b) - .1*(center)$);
\coordinate (cplus) at ($1.1*(c) - .1*(center)$);
\coordinate (dplus) at ($1.1*(d) - .1*(center)$);
\coordinate (x) at  ($.6*(a) +.15*(b)$);
\coordinate (y) at  ($.6*(a) +.15*(c)$);
\coordinate (z) at  ($.6*(c) +.15*(a)$);
\coordinate (w) at  ($.6*(c) +.15*(b)$);
\coordinate (u) at  ($.6*(b) +.15*(c)$);
\coordinate (v) at  ($.6*(b) +.15*(a)$);
\coordinate (rc) at($.33*(a)+.33*(b)+.33*(d)$);
\coordinate (ar) at ($.667*(a)-.33*(b)-.33*(d)$); 
\coordinate (br) at ($.667*(b)-.33*(a)-.33*(d)$); 
\coordinate (cr) at ($.667*(d)-.33*(b)-.33*(a)$); 
\coordinate (xr) at  ($.6*(ar) +.15*(br)+(rc)$);
\coordinate (yr) at  ($.6*(ar) +.15*(cr)+(rc)$);
\coordinate (zr) at  ($.6*(cr) +.15*(ar)+(rc)$);
\coordinate (wr) at  ($.6*(cr) +.15*(br)+(rc)$);
\coordinate (ur) at  ($.6*(br) +.15*(cr)+(rc)$);
\coordinate (vr) at  ($.6*(br) +.15*(ar)+(rc)$);
\coordinate (bkc) at  ($.33*(a)+.33*(c)+.33*(d)$);
\coordinate (abk) at ($.667*(a)-.33*(c)-.33*(d)$); 
\coordinate (bbk) at ($.667*(c)-.33*(a)-.33*(d)$); 
\coordinate (cbk) at ($.667*(d)-.33*(c)-.33*(a)$); 
\coordinate (xbk) at  ($.6*(abk) +.15*(bbk)+(bkc)$);
\coordinate (ybk) at  ($.6*(abk) +.15*(cbk)+(bkc)$);
\coordinate (zbk) at  ($.6*(cbk) +.15*(abk)+(bkc)$);
\coordinate (wbk) at  ($.6*(cbk) +.15*(bbk)+(bkc)$);
\coordinate (ubk) at  ($.6*(bbk) +.15*(cbk)+(bkc)$);
\coordinate (vbk) at  ($.6*(bbk) +.15*(abk)+(bkc)$);
\coordinate (flc) at  ($.33*(b)+.33*(c)+.33*(d)$);
\coordinate (afl) at ($.667*(b)-.33*(c)-.33*(d)$); 
\coordinate (bfl) at ($.667*(c)-.33*(b)-.33*(d)$); 
\coordinate (cfl) at ($.667*(d)-.33*(c)-.33*(b)$); 
\coordinate (xfl) at  ($.6*(afl) +.15*(bfl)+(flc)$);
\coordinate (yfl) at  ($.6*(afl) +.15*(cfl)+(flc)$);
\coordinate (zfl) at  ($.6*(cfl) +.15*(afl)+(flc)$);
\coordinate (wfl) at  ($.6*(cfl) +.15*(bfl)+(flc)$);
\coordinate (ufl) at  ($.6*(bfl) +.15*(cfl)+(flc)$);
\coordinate (vfl) at  ($.6*(bfl) +.15*(afl)+(flc)$);
\fill[red!10]  (vfl) to (xfl) to (yfl) to (ur) to (wr) to (zr) to (yr) to (ybk) to (xbk) to (vbk) to  (z) to (w) to (vfl);
 \begin{scope}[xshift=.4cm, yshift=.4cm]
 \draw [black!50] (0,0) .. controls (0:.75) and (90:.75)   .. (0,0);
\draw  [black!50](0,0) .. controls (90:.75) and (180:.75)   .. (0,0);
\draw  [black!50](0,0) .. controls (180:.75) and (270:.75)   .. (0,0);
\draw  [black!50](0,0) .. controls (270:.75) and (360:.75)   .. (0,0);
\end{scope}
\draw [red, densely dotted]  (xfl) to (yfl) to (zfl) to (wfl) to (ufl) to (vfl) to (xfl); 
\draw [red, densely dotted]  (xbk) to (ybk) to (zbk) to (wbk) to (ubk) to (vbk) to (xbk); 
\draw [red]  (xr) to (yr) to (zr) to (wr) to (ur) to (vr) to (xr); 
\draw [red]  (x) to (y) to (z) to (w) to (u) to (v) to (x); 
\draw  [red] (u) to (xfl);
\draw [red] (v) to (vr);
\draw [red] (w) to (vfl);
\draw [red] (x) to (xr);
\draw [red] (y) to (xbk);
\draw [red] (z) to (vbk);
\draw [red] (ur) to (yfl);
\draw [red, densely dotted] (wr) to (zfl);
\draw [red] (yr) to (ybk);
\draw [red, densely dotted] (zr) to (zbk);
\draw [red, densely dotted] (ubk) to (ufl);
\draw [red, densely dotted] (wbk) to (wfl);
\draw [red] (vbk) to (xbk) to (ybk);
\draw [red] (vfl) to (xfl) to (yfl);
\draw (cplus) to (aplus) to (dplus) to (bplus) to  (cplus);
\draw[dotted] (cplus) to (dplus);
\draw (aplus) to (bplus);
  \end{tikzpicture}
\end{center}
\caption{Permutohedra in rose faces for $n=3$ and $n=4$}
\label{fig:permutohedra}
\end{figure}

If $G$ is a rose, then every proper subset of edges is a core subgraph, so every proper face of $\csigma(G)$ is shaved.  The resulting convex polytope $J(G)$ is called 
a {\em permutohedron} of rank $n$.  The cases $n=3$ and $n=4$ are illustrated in Figure~\ref{fig:permutohedra}.

An example of $J(G)$ when $G$ is not a rose is shown in Figure~\ref{fig:jewel}. Here a face $\sigma(H)$ of $\sigma(G)$ is identified with the subset of $\{0,1,2,3\}$ indexing the edges of $H$. The core subgraphs properly contained in $G$  are spanned by the sets $\{e_0\}$, $ \{e_1\}$, $\{e_2,e_3\}$, $\{e_0,e_1\}$, $\{e_1,e_2,e_3\}$ and $\{e_0,e_2,e_3\}$, so the faces that are shaved are  $\{1,2,3\}$, $ \{0,2,3\}$, $\{0,1\}$, $\{2,3\}$, $ \{0\}$ and $\{1\}$ respectively.   Faces of $\csigma(G)$  obtained by collapsing a maximal tree are called {\em rose faces}; note that these are not shaved.  In our example $\csigma(G)$  has two rose faces, $e_2=0$ and $e_3=0$.  Since the rank of $G$ is $3$, each rose face contains a permutohedron with six vertices.  These are the only vertices of $J(G),$ i.e. $J(G)$ is the convex hull of the permutohedra contained in the rose faces of $\csigma(G)$.   This description of $J(G)$ is general:

\begin{figure}
\begin{center}
\begin{tikzpicture}[scale=1] 
\fill [red!10] (105:1.1) to (60:1.9) to (35:2.5) to (25:2.5) to (-4:1.85) to (315:1.1) to (225:1.1) to (195:1.1) to (105:1.1);
\draw (90:2) to (210:2) to (330:2) to (90:2);
\draw (90:2) to (30:3) to (330:2);
\draw [dotted] (210:2) to (30:3);
\draw [red] (75:1.1) to (105:1.1) to (195:1.1) to (225:1.1) to (315:1.1) to (345:1.1) to (75:1.1);
\draw [red] (60:1.5) to (60:1.9) to (35:2.5) to (25:2.5) to (-4:1.85) to (-4:1.5) to (60:1.5);
\draw [red] (75:1.1) to (60:1.5);
\draw [red] (105:1.1) to (60:1.9);
\draw [red] (-4:1.5) to (345:1.1);
\draw [red] (-4:1.85) to (315:1.1);
\draw [red,densely dotted] (225:1.1) to (25:2.5);
\draw [red,densely dotted] (195:1.1) to (35:2.5);
\draw (0,0) .. controls (45:.75) and (135:.75)   .. (0,0);
\draw (0,0) .. controls (15:.75) and (285:.75)   .. (0,0);
\draw (0,0) .. controls (255:.75) and (165:.75)   .. (0,0);
\begin{scope}[xshift=1.6cm, yshift=.85cm, rotate=45]
\draw (0,0) .. controls (45:.75) and (135:.75)   .. (0,0);
\draw (0,0) .. controls (15:.75) and (285:.75)   .. (0,0);
\draw (0,0) .. controls (255:.75) and (165:.75)   .. (0,0);
\end{scope}
\node (a0) at (0,2.3) {$1$};
\node (b0) at (-2,-1.3) {$2$};
\node (c0) at (2,-1.3) {$0$};
\node (d0) at (2.9,1.7) {$3$};
\begin{scope}[xshift=-6cm]
\draw (.75,0) .. controls (2,1.25) and (2,-1.25)   .. (.75,0);
\draw (-.75,0) .. controls (-2,1.25) and (-2,-1.25)   .. (-.75,0);
\draw (-.75,0) .. controls (-.5,.5) and (.5,.5)   .. (.75,0);
\draw (.75,0) .. controls (.5,-.5) and (-.5,-.5)   .. (-.75,0);
\node (a) at (0,.6) {$e_2$};
\node (b) at (0,-.6) {$e_3$};
\node (c) at (2,0) {$e_1$};
\node (d) at (-2,0) {$e_0$};
\end{scope}
 \end{tikzpicture}
\end{center}
\caption{Example of a graph which is not a rose and its associated jewel}
\label{fig:jewel}
\end{figure}
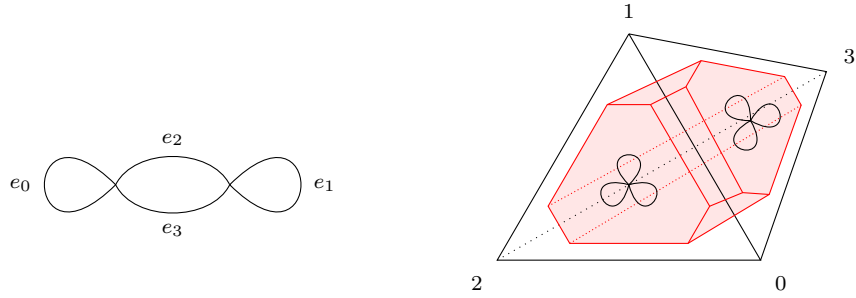

\begin{proposition}\label{prop:rosefaces} For any $(G,g)$ in $\On,$ the jewel $J(G)$ is the convex hull of the permutohedra contained in the rose faces of $\csigma(G).$
\end{proposition}

\begin{proof}
Since $J(G)$ is a convex polytope, it is the convex hull of its vertices. Each vertex lies in the interior of some (not necessarily proper) face $\tau$ of $\csigma(G),$ so it suffices to show that $\tau$ must be a rose face. The face $\tau$ corresponds to a subgraph $H\subseteq G$, i.e. $\tau=\csigma(H)$.  If $\tau$ were at infinity, then the complement $H^c$ would contain a non-trivial core subgraph $C$.  Since the face $\csigma(C^c)$ opposite $\csigma(C)$ is truncated when forming $J(G)$ and  $H\subseteq C^c,$ the face  $\tau=\csigma(H)\subseteq\csigma(C^c)$ would not intersect $J(G)$, so would not contain any vertices of $J(G)$.  Therefore $\tau$ is not at infinity, i.e. $H^c$ is a forest and $\csigma(G\sslash H^c)<\sigma(G)$.  If we choose the constants carefully we in fact have $J(G\sslash H^c)=J(G)\cap \tau.$  This reduces the problem to showing that for any graph $G$, if $J(G)$ has a vertex in the interior of $\csigma(G)$ then $G$ is a rose.

In the following we identify faces of $\sigma(G)$ with subsets of $\{0,\ldots,m\}$, and call such a subset {\em core} if it corresponds to a core subgraph. Suppose $y=(y_0,\ldots,y_m)$ is a vertex of $J(G)$ with all $y_i$ positive.  Then $y$ lies on bounding hyperplanes $\mathcal H_A$  for some collection $\mathcal S(y)$ of core subsets $A\subset \{0,\ldots,m\}$, i.e. $y$ satisfies $\sum_{i\in A}y_i=c_{A}$ for all $A\in \mathcal S(y)$.  

Note that the union of two core subgraphs  is always a core subgraph. Thus if $A$ and $B$ are in $\mathcal S(y),$ then $(A\cup B)^c$ is shaved. So for all $x\in J(G)$ 
 $$\sum_{k\in A\cup B}x_k\geq c_{A\cup B}.$$
Since $\sum_{i\in A}y_i= c_{A}$ and $\sum_{j\in B}y_j= c_{B}$ we have
\[
  \sum_{k\in A\cup B}y_k\leq c_{A}+ c_{B}.
\]
If $A$ and $B$ are both proper subsets of $A\cup B$, then each of $c_A$ and $c_B$ is less than half of $c_{A\cup B}$, so
\(
  c_{A} + c_{B} < c_{A \cup B}
\),
giving
$$ \sum_{k\in A\cup B}y_k < c_{A\cup B},$$  This contradicts the assumption that $y$ is in $J(G).$

Thus the subsets of $\mathcal S(y)$ are nested, i.e. form a flag.  Since a vertex is the intersection of at least $m$ hyperplanes, $\mathcal S(y)$ contains at least $m$ proper subsets, so up to permutation the flag is $\{0\},\{0,1\},\{0,1,2\}\ldots,\{0,1,\ldots m-1\}$.  
Thus $\{e_0\}$ is a core subgraph, so it must be a loop.  Since $\{e_0,e_1\}$ also forms a core subgraph, and core subgraphs have no separating edges, $e_1$ is also a loop.  Continuing, we get that $e_0,\ldots,e_{m-1}$ are all loops, which implies that $e_{m}$ too is a loop, since $G$ has no separating edges.  Since $G$ is connected and all edges are loops, we conclude that $G$ is a rose (and $n=m+1$).
\end{proof}
 
\subsection{Fitting jewels together to form $\BVn$}\label{fitting}
Suppose $(G',g')$ can be obtained from $(G,g)$ by collapsing the edges of some subforest  $\Phi$ of $G$ to points.  Then for appropriate truncating constants, $J(G',g')$ is a face of $J(G,g)$. Specifically, we need  $c_{A'}=c_A$ where $A=\core(A'\cup\Phi)$  is the largest core graph in $G$ mapping to $A'$. To make the constants $c_A$ consistent over all $J(G,g)$ containing $J(G',g')$, we assume that $c_{A}$  depends only the rank of $H_1(A)$. This works  because $\rank H_1(A')=\rank H_1(A)$ and if $A \subsetneq B$ are core graphs  one needs to remove at least one edge of $B$ to get $A$; this decreases the rank of $H_1(B)$ since $B$ has no separating edges.

Let $\BVn$ denote the union of the cells $J(G,g)$ for all marked graphs $(G,g)$ in $\On$.  We claim that this a closed subspace of $\On$ which is a deformation retract. To see this, recall that $\On$ is the union of the simplices $\osigma(G,g)$ glued together using the   same face relations,  each cell $J(G,g)$ is evidently a deformation retract of the closure $\cOsigma(G,g)$ of $\osigma(G,g)$ in $\On$, and the deformation retraction can be taken to restrict to deformation retractions of all faces $\cOsigma(G',g')$.    Figure~\ref{OS2} shows the relation between the spaces $\On, \BVn$ and the spine $K_n$ of $
\On$ for the case $n=2$. Here the Euclidean simplices have been deformed for artistic convenience so that they fit into the Poincar\'e disk as hyperbolic triangles. 

The following statement is an immediate corollary of Proposition~\ref{prop:rosefaces}.

\begin{corollary}\label{cor:vertices}
The vertices of $\BVn$ are the {\em marked ordered roses} $(g,R,e_1,\ldots,e_n)$, where   $(R,g)$  is a marked rose and $e_1,\ldots,e_n$ is an (ordered) list of the petals of $R$.
\end{corollary}
 
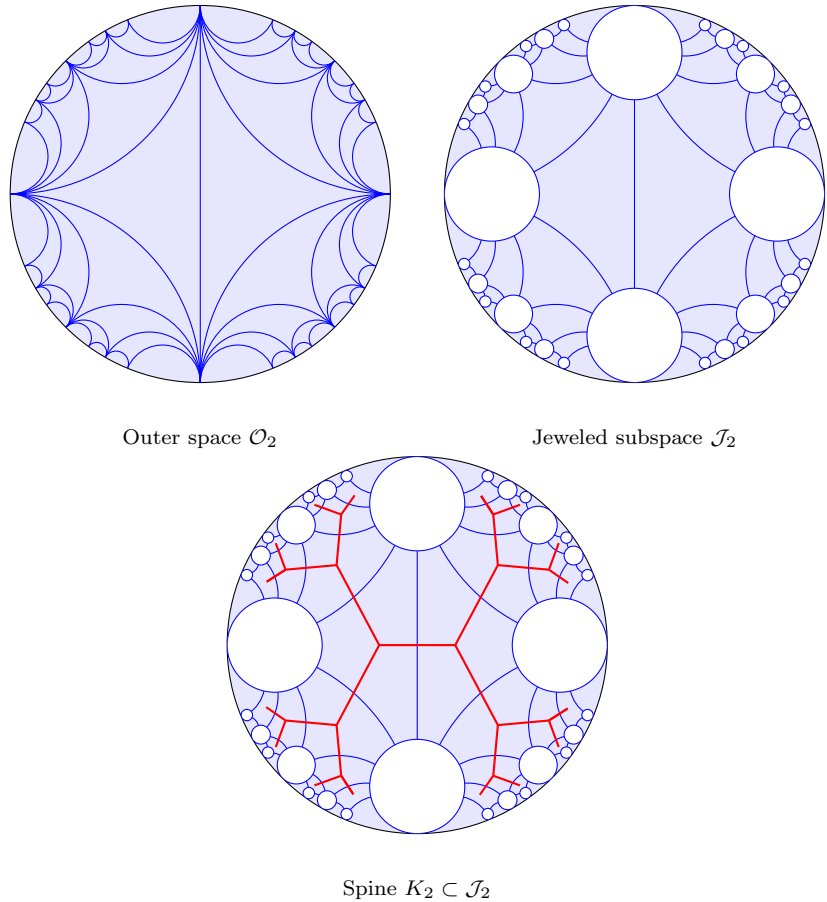
\begin{figure}
\begin{center}
\begin{tikzpicture}[scale=2.5]
\draw [fill=blue!10] (0,0) circle (1);
\begin{scope}
\clip (0,0) circle (1);
\hgline{0}{180};
\foreach \x in {0,...,3}
\hgline {0+90*\x}{90+90*\x};
\foreach \x in {0,...,7}  
\hgline {0+45*\x}{45+45*\x};
\foreach \x in {0,...,3} 
{
\hgline {0+90*\x}{30+90*\x}; 
\hgline {30+90*\x}{45+90*\x};
\hgline {45+90*\x}{60+90*\x}; 
\hgline {60+90*\x}{90+90*\x};
}
\foreach \x in {0,...,3} 
{
\hgline {0+90*\x}{22.5+90*\x}; 
\hgline {22.5+90*\x}{30+90*\x}; 
\hgline {30+90*\x}{36+90*\x};
\hgline {36+90*\x}{45+90*\x};
\hgline {45+90*\x}{54+90*\x};
\hgline {54+90*\x}{60+90*\x};
\hgline {60+90*\x}{67.5+90*\x};
\hgline {67.5+90*\x}{90+90*\x};
}
\end{scope}
\node (a) at (0,-1.3) {Outer space $\mathcal O_2$};
\end{tikzpicture} 
\qquad
\begin{tikzpicture}[scale=2.5]
\draw [fill=blue!10] (0,0) circle (1);
\begin{scope}
\clip (0,0) circle (1);
\hgline{0}{180};
\foreach \x in {0,...,3}
\hgline {0+90*\x}{90+90*\x};
\foreach \x in {0,...,7}  
\hgline {0+45*\x}{45+45*\x};
\foreach \x in {0,...,3} 
{
\hgline {0+90*\x}{30+90*\x}; 
\hgline {30+90*\x}{45+90*\x};
\hgline {45+90*\x}{60+90*\x}; 
\hgline {60+90*\x}{90+90*\x};
};
\foreach \x in {0,...,3} 
{
\hgline {0+90*\x}{22.5+90*\x}; 
\hgline {22.5+90*\x}{30+90*\x}; 
\hgline {30+90*\x}{36+90*\x};
\hgline {36+90*\x}{45+90*\x};
\hgline {45+90*\x}{54+90*\x};
\hgline {54+90*\x}{60+90*\x};
\hgline {60+90*\x}{67.5+90*\x};
\hgline {67.5+90*\x}{90+90*\x};
}
\foreach \x in {0,...,3} 
{
\draw [blue, fill=white] (0+90*\x:.75) circle (.25); 
\draw [blue, fill=white] (45+90*\x:.9) circle (.1); 
\draw [blue, fill=white] (30+90*\x:.95) circle (.05); 
\draw [blue, fill=white] (60+90*\x:.95) circle (.05); 
\draw [blue, fill=white] (36+90*\x:.97) circle (.03);
\draw [blue, fill=white] (22.5+90*\x:.97) circle (.03);  
\draw [blue, fill=white] (54+90*\x:.97) circle (.03);  
\draw [blue, fill=white] (67.5+90*\x:.97) circle (.03);  
}
\end{scope}
\node (b) at (0,-1.3) {Jeweled subspace $\mathcal J_2$};
\end{tikzpicture} 
\qquad
\begin{tikzpicture}[scale=2.5]
\draw [fill=blue!10] (0,0) circle (1);
\begin{scope}
\clip (0,0) circle (1);
\hgline{0}{180};
\foreach \x in {0,...,3}
\hgline {0+90*\x}{90+90*\x};
\foreach \x in {0,...,7}  
\hgline {0+45*\x}{45+45*\x};
\foreach \x in {0,...,3} 
{
\hgline {0+90*\x}{30+90*\x}; 
\hgline {30+90*\x}{45+90*\x};
\hgline {45+90*\x}{60+90*\x}; 
\hgline {60+90*\x}{90+90*\x};
}
\foreach \x in {0,...,3} 
{
\hgline {0+90*\x}{22.5+90*\x}; 
\hgline {22.5+90*\x}{30+90*\x}; 
\hgline {30+90*\x}{36+90*\x};
\hgline {36+90*\x}{45+90*\x};
\hgline {45+90*\x}{54+90*\x};
\hgline {54+90*\x}{60+90*\x};
\hgline {60+90*\x}{67.5+90*\x};
\hgline {67.5+90*\x}{90+90*\x};
}
\foreach \x in {0,...,3} 
{
\draw [blue, fill=white] (0+90*\x:.75) circle (.25); 
\draw [blue, fill=white] (45+90*\x:.9) circle (.1); 
\draw [blue, fill=white] (30+90*\x:.95) circle (.05); 
\draw [blue, fill=white] (60+90*\x:.95) circle (.05); 
\draw [blue, fill=white] (36+90*\x:.97) circle (.03);
\draw [blue, fill=white] (22.5+90*\x:.97) circle (.03);  
\draw [blue, fill=white] (54+90*\x:.97) circle (.03);  
\draw [blue, fill=white] (67.5+90*\x:.97) circle (.03);  
}
\draw [red, thick] (0:-.2) to (0:.2);
\draw [red, thick] (0:.2) to (45:.6);
\draw [red, thick] (0:.2) to (-45:.6);
\draw [red, thick] (0:-.2) to (135:.6);
\draw [red, thick] (0:-.2) to (-135:.6);
\foreach \x in {0,...,3}
{
\draw [red, thick] (45+\x*90:.6) to (60+\x*90:.8);
\draw [red, thick] (45+\x*90:.6) to (30+\x*90:.8);
\draw [red, thick] (60+\x*90:.8) to (67+\x*90:.86);
\draw [red, thick] (60+\x*90:.8) to (67+\x*90:.86);
\draw [red, thick] (30+\x*90:.8) to (36+\x*90:.92);
\draw [red, thick] (30+\x*90:.8) to (22.5+\x*90:.86);
\draw [red, thick] (60+\x*90:.8) to (54+\x*90:.92);
}
\end{scope}
\node (c) at (0,-1.3) {Spine $K_2\subset \mathcal J_2$};
\end{tikzpicture} 
\caption{Outer space, the subspace of jewels and the spine, for $n=2$}\label{OS2}
\end{center}
\end{figure}

\section{Simplicial completion of $\BVn$ and a further retract}

We next want to investigate the connectivity properties of $\BVn$  at infinity, in particular to prove

\begin{theorem}\label{thm:Jmain} The space $\BVn$ is $(2n-5)$-connected at infinity.  
\end{theorem}

 For this it is convenient to replace $\BVn$ by a simplicial complex so that we may use simplicial Morse theoretic arguments.   
We do this in two stages.  First, replace each cell $J(G,g)$ by a simplex $s(G,g)$ with the same vertices to obtain a simplicial complex $\bS_n$ (this was also done in \cite{BF}, at the end of the paper).  

\begin{lemma} $\bS_n$ is contractible, and $\Outn$ acts properly and cocompactly.
\end{lemma}

\begin{proof}
The spaces $\bS_n, \BVn$ and $\On$ are all complexes of contractible spaces (i.e. the spaces $s(G,g), J(G,g)$ and $\cOsigma(G,g)$ respectively) and all have the same nerve, namely the spine $K_n$.  Since $K_n$ is contractible,  $\bS_n, \BVn$ (and $\On$) are also contractible (see, e.g., \cite{Hatcher}, section 4.G).
The action permutes the cells of $\bS_n$ and $\BVn$, and there is a compact fundamental domain since $K_n$ is cocompact  and the cells of $\bS_n$ and $\BVn$ are compact.
\end{proof}

Now let $\bR_n$ be the simplicial complex whose vertices are the marked roses in $\On$, and where roses are in the same simplex if and only if they  can be obtained from a common $(G,g)$ by collapsing  maximal trees.

\begin{corollary} $\bR_n$ is contractible, and the action of $\Outn$ on it is proper and cocompact.
\end{corollary}
        \begin{proof} The map sending vertices of $\bS_n$ to vertices $\bR_n$  by forgetting the ordering on the petals of the rose extends to a simplicial map and the inverse image of each simplex is itself a simplex, so the map is a homotopy equivalence (see, e.g., \cite{HV}, Corollary 2.7).
\end{proof}

Since both $\BVn$ and  $\bR_n$ are contractible with proper cocompact $Out(F_n)$-actions,  Theorem~\ref{thm:Jmain} is equivalent to the following.

\begin{theorem}\label{thm:Rmain} The simplicial complex $\bR_n$ is $(2n-5)$-connected at infinity.  
\end{theorem}

\section{The Morse function  and ascending links}\label{Morse} 

We are now ready to attack connectivity at infinity, using Morse theory.  We begin by defining a {\em Morse function} $\mu$ on the vertices of $\bR_n$ with values in a certain ordered abelian group.   

\subsection{The Morse function}

Let $\rho=(R,g)$ be a vertex of $\bR_n,$ i.e. a marked rose, and suppose $F_n$ is generated by $x_1,\ldots,x_n.$ Let $\cW_0$ be the set of conjugacy classes of elements of the form $x_i$ or $x_ix_j$ or $ x_ix_j^{-1}$ for $i\neq j$, and let $\cW=\{w_1,w_2,w_3,\ldots\}$  be a list of {\em all} the  conjugacy classes in $F_n$.  

Given a conjugacy class $w$ and an edge $e\in R$, define $|e|_w$ to be the minimum, over all loops $\gamma$ homotopic to $g(w)$, of the number of times $\gamma$ crosses $e$ in either direction.   Then define $$|\rho|_w=\sum_{e\in R} |e|_w$$
and $$|\rho|_0=\sum_{w\in \cW_0} |\rho|_w.$$   Finally, define 
$$\mu(\rho)=(|\rho|_0,|\rho|_{w_1}, |\rho|_{w_2},\ldots) \in \R\times \R^{\cW}$$
Here $\R\times \R^{\cW}$ is an ordered abelian group with the lexicographical order.

\begin{lemma}\label{well}
\begin{enumerate}
\item $|\rho|_0>0.$
\item $|\rho|_{w_i}>0$ for all $i$.
\item If $\mu(\rho)=\mu(\rho')$ then $\rho=\rho'.$
\item At most finitely many $\rho$ have $\mu(\rho)\leq N$ for a given $N.$
\item $\mu$ well-orders the vertices of $\bR_n$.
\end{enumerate}
\end{lemma}

\begin{proof}
A proof may be found in \cite{V}. It relies on the fact that a free action on a simplicial tree is uniquely determined by its translation length function, which was proved by Culler and Morgan \cite[Theorem~3.7]{CM} and independently by Alperin and Bass \cite[Theorem~7.13]{AB}. The proof from \cite{V} applies here because $|\rho|_0$ is the norm used in the original proof in \cite{CV} of the contractibility of Outer space.
\end{proof}

\subsection{Ascending links}

In this section we reduce connectivity of $\bR_n$ at infinity to a local problem. To do this we use the Morse function to arrange all vertices of $\bR_n$ into an ordered list. The  link of each vertex $v$ then has a descending part (spanned by the vertices listed before $v$) and an ascending part (spanned by the vertices listed after $v$). A standard argument shows that $\bR_n$ is $(2n-5)$-connected at infinity provided that the ascending subcomplex of each vertex link is $(2n-5)$-connected.  Here are the details.

{%
\newcommand{\Height}{h}%
\newcommand{\mapcolon}{\colon}%
\newcommand{\NNN}{\mathbb{N}}%
\newcommand{\Level}{i}%
By Lemma~\ref{well}, the map
\begin{align*}
  \Height \mapcolon \bR_n^{(0)} & \longrightarrow \NNN \\
  v & \mapsto \# \{ u \in \bR_n^{(0)} |\, \mu(u) < \mu(v) \} 
\end{align*}
is a well-defined bijection between the $0$-skeleton $\bR_n^{(0)} \subseteq \bR_n$ and the natural numbers $\NNN$; it just counts the number of vertices that come before $v$ in the well-ordering of roses given by $\mu$. Thus we have a list
\(
  \rho_0, \rho_1, \rho_2, \ldots
\)
of all the roses in ascending order of $\mu$-values.

 Let $\bR_{\geq \Level}$ be the subcomplex of $\bR_n$ spanned by the vertices $\rho_{\Level}, \rho_{\Level+1}, \rho_{\Level+2}, \ldots$.  The {\em ascending link}
 of the rose $\rho_{\Level}$  is defined to be
$$  \lk^+(\rho_{\Level}) :=
  \lk(\rho_{\Level}) \cap \bR_{\geq \Level+1}
$$

Recall that a simplicial complex is {\em $k$-spherical}   if it is $k$-dimensional and $(k-1)$-connected.  
In Section~\ref{sec:proof} we will prove 
 \begin{theorem}\label{thm:link} For every rose $\rho$, the ascending link $\lk^+(\rho)$ is $(2n-4)$-spherical.
 \end{theorem}
 
From this we can deduce Theorem~\ref{thm:Rmain} using the following argument from \cite[Theorem~5.3]{BF}. Since every compact subset of $\bR_n$ is disjoint from $\bR_{\geq \Level}$ for sufficiently large $\Level$, it suffices to show that each $\bR_{\geq \Level}$ is $(2n-5)$-connected. This is done by induction.

The base case $\Level=0$ is the statement that   $\bR_n$ is contractible. 
Assuming that $\bR_{\geq \Level}$ is $(2n-5)$-connected, observe that $\bR_{\geq \Level}$ is obtained from $\bR_{\geq \Level+1}$ by adding the vertex $\rho_{\Level}$ and coning off its  $(2n-5)$-connected ascending link $\lk(\rho_{\Level}) \cap \bR_{\geq \Level+1}$. By the theorems of Hurewicz and van-Kampen, it follows that $\bR_{\geq \Level+1}$ is again $(2n-5)$-connected.}

\section{Blowups and ideal edges}

It remains to prove Theorem~\ref{thm:link}.  For the proof we use the technology introduced in \cite{CV} relating graphs, maximal trees, roses and the norm $\mu$. In this section we give a brief review of this technology.  A more careful discussion can be found in \cite{V}.

Let $G$ be a graph of rank $n$, let  $t_1,\ldots,t_k$ be the edges of a maximal tree $T$ in $G$ and let $e_1,\ldots,e_n$ be the remaining edges of $G$. Choose an orientation for each $e_i$ and let $\einv_i$ denote the same edge with the opposite orientation.  Removing any $t_i$ cuts $T$ into two subtrees (either of which may be a point), and determines a partition $\alpha_i$ of  the set $E=\{e_1,\einv_1,\ldots,e_n,\einv_n\}$ into two pieces according to which subtree contains the terminus of the oriented edge. Partitions   $\alpha_i$ and $\alpha_j$ determined by $t_i$ and $t_j$ are {\em compatible} in the sense that one side of $\alpha_i$ is disjoint from one side of $\alpha_j$.  

Collapsing $T$ produces a rose $R_n$   with oriented petals $E=\{e_1,\overline e_1,\ldots,e_n,\overline e_n\}$.  The graph $G$ can be uniquely reconstructed from the set of partitions $\{\alpha_j\}$ of $E$. 

Now let $\rho=(R,g)$ be a marked rose, and 
let $E=\{e_1,\einv_1,e_2,\einv_2,\ldots,e_n,\einv_n\}$ be the oriented petals of $R$.  A partition of $E$ into two parts $A$ and $A^c=E-A$ {\em splits $e_i$} if $e_i$ and $\einv_i$ are on different sides of the partition. An {\em ideal edge} of $\rho$ is a partition of $E$ into two sets, each with at least two elements, which splits some $e_i$.    A {set of pairwise-compatible} ideal edges is called an {\em ideal tree}; it corresponds to a maximal tree in a graph that has no leaves or bivalent vertices and no separating edges.  

Note that an ideal edge $\alpha=(A,A^c)$ is determined by either of its sides, which we call {\em representatives} for $\alpha$.

\subsection {Ideal edges, star graphs and $\mu$}
 
Let $R$ be a rose and $E$ its set of oriented edges, equipped with the involution $e\mapsto \overline e\,$ sending $e$ to the same edge with opposite orientation. If $\gamma=a_1\ldots a_k$ is a cyclically reduced edge-path  in $R$ the {\em star graph} of $\gamma$ is the graph with vertices $E$ and an edge from $a_i$ to $\overline a_{i+1}$ for each $i=1,\ldots,k$ (setting $k+1=1$).

Now fix a marking $g\colon R_n\to R.$  Each conjugacy class $w$ of $F_n$ can be represented by an edge path in $R_n$,  and  its image $g(w)$ is homotopic to a unique cyclically reduced edge path $\gamma(w)$.  We define  $st(w)$  to be the star graph of $\gamma(w)$,  
and $st(\cW_0)$ to be the superposition of $st(w)$ for all $w\in \cW_0$.   
Thus the sequence $(\cW_0,w_1,w_2,w_3,\ldots)$ gives an infinite sequence of star graphs associated to $\rho=(R,g)$.

For disjoint subsets $X$ and $Y$ of $E$, define $(X\ndot Y)_w$ to be the number of edges of $st(w)$ with one vertex in $X$ and one vertex in $Y$, and $$X\ndot Y=((X\ndot Y)_{\cW_0},(X\ndot Y)_{w_1},(X\ndot Y)_{w_2},\ldots)\in \R\times \R^{\cW}.$$
If $X\subset E,$ define $|X|_w=(X\ndot X^c)_w$ and $|X|=X\ndot X^c$.  In particular, if $X=\{e\}$ this agrees with our previous definition of $|e|_w$.  

Note that the dot product is commutative:
\[
  X \ndot Y = Y \ndot X
\]
and the following ``distributive law" holds for pairwise disjoint subsets $X,Y,Z\subseteq E$:
\[
  ( X  \sqcup Y ) \ndot Z = X \ndot Z + Y \ndot Z
\]
If $\alpha$ is an ideal edge with sides $A$ and $A^c$, define $|\alpha|=|A|=|A^c|$.  

\begin{lemma}[Positivity of the dot product and absolute value]\label{dot-and-norm}
For any non-empty disjoint subsets $X,Y\subseteq E$, we have $X \ndot Y > 0$. In particular:
\begin{enumerate}
\item $|e|>0$ for all $e \in E$.
\item $e\ndot f>0$ for all $e\neq f$ in $E$
\item $|\alpha| > 0$ for all ideal edges $\alpha$.
\end{enumerate}
\end{lemma}

\begin{proof}
By distributivity, we can reduce to singletons $X=\{e\}$ and $Y=\{f\}$. Let $w_e$ be the conjugacy class represented by the loop $e$ in $\rho$. Define $w_f$ analogously.

If $f=\overline e$ then $st(w_e)$ is a single edge from $e$ to $f$, so the $w_e$-coordinate of $e\ndot f$ is equal to $1 \neq 0$.

Otherwise $st(w_ew_f^{-1})$ consists of two edges, one joining $e$ and $f$ and one joining $\einv$ and $\finv$. Therefore the  $w_ew_f^{-1}$ coordinate of $e\ndot f=1\neq0$.
\end{proof}

\begin{lemma}[Consequences of equality]\label{consequences-of-equality}\ 
\begin{enumerate}
\item If $|e|=|f|$ then $f=\overline e$ or $f=e$.
\item If $\alpha, \beta$ are two distinct ideal edges then $|\alpha|\neq |\beta|$.
\item If $\alpha$ is an ideal edge and $e\in E$ then $|\alpha|\neq |e|$.
\end{enumerate}
\end{lemma}
\begin{proof}
Let $w_e$ denote the conjugacy class corresponding to the loop $e$.
\begin{enumerate}
\item  If $f\not\in\{e,\einv\}$, then the $w_e$ coordinates of $|e|$ and $|f|$ are different.

\item If there is an edge $e$ split by only one of $\alpha$ or $\beta$ then the $w_e$-coordinates of $\alpha$ and $\beta$ are different.   If not, choose sides $A$ and $B$ so that $A\cap B, A-(A\cap B)$ and $B-(A\cap B)$ are all non-empty.

If there are $e\in A\cap B$ and $f\in A-(A\cap B)$ which are both split, then $st(ef)$ crosses $\alpha$ but not $\beta$.
If there are $e\in A\cap B$ and $f\in A-(A\cap B)$ neither of which is split, then $st(ef)$ crosses $\beta$ but not $\alpha$.  In either case   the $w_{ef}$ coordinates  of $|\alpha|$ and $|\beta|$ are different.

If  $e\in A\cap B$ is split but $f\in A-(A\cap B)$ is not, then choose $z\in B-(B\cap A)$. If $z$ is split, we are in a previous case by symmetry, so we may assume $\overline z\in B-(A\cap B)$.  Then $st(fez\overline e)$ crosses $\alpha$ but not $\beta$.

It remains to consider the case that $e\in A\cap B$ is not split but $f\in A-(A\cap B)$ is split.  Since  $B^c$ cannot be a singleton, we can find  $z\neq f$ in  $A-(A\cap B)$ or in $(A\cap B)^c$.      In the first case we may assume $\overline z\in B-(A\cap B)$; otherwise we can reduce to a previous case by exchanging the roles of $z$ and $f$. Then  $st(ef\overline z)$ crosses $\alpha$ but not $\beta$. In the second case we may assume $\overline z\in (A\cap B)^c$; otherwise we could again reduce to a previous case by exchanging $f$ and $z$. Then $st(efz\overline f)$ crosses $\alpha$ but not $\beta$.

\item If $\alpha$ doesn't split $e$ then the $w_e$-coordinates of $|e|$ and $|\alpha|$ are different.  If $\alpha$ splits both $e$ and $f$, then the $w_{ef}$-coordinate is different. If $e$ is the only edge split by $\alpha$, then choose $f,\finv$  on one side of $\alpha$ and $h, \overline h$ on the other side.  Then the coordinate of $w_{fh}$ is different.
\end{enumerate}
\end{proof}

\section{Proof of Theorem~\ref{thm:link}}\label{sec:proof}

In this section we first show that the ascending link $\lk^+(\rho)$ is homotopy equivalent to a simplicial complex $\bZ(\rho)$ whose vertices are certain ideal edges of $\rho$, then  we prove that  $\bZ(\rho)$  is $(2n-4)$-spherical in  Theorem~\ref{thm:main}.

\subsection{Tools of the trade}

We recall here a few standard tools which will be useful in our proof.  We start with some elementary observations about  $k$-spherical complexes. For details see, e.g., \cite{Q}.

\begin{lemma} A $k$-spherical complex is either contractible or homotopy equivalent to a nontrivial wedge of $k$-spheres. 
\end{lemma}

\begin{lemma}
Let $K$ be a $k$-spherical subcomplex of a simplicial complex $L$ and $v$ a vertex of $L$.  If $\lk(v)\cap K$ is $(k-1)$-spherical, then the subcomplex spanned by $K$ and $v$ is $k$-spherical.  
\end{lemma}

\begin{lemma}
If $K$ is $k$-spherical and $L$ is $\ell$-spherical, then the simplicial join $K\ast L$ is $(k+\ell+1)$-spherical.
\end{lemma}

We will use a simple version of Quillen's Poset Lemma. A map $f\colon X\to Y$ between posets is a {\em poset map} if $x\leq x'$ implies $f(x)\leq f(x')$ for all $x,x'\in X$. The {\em geometric realization} of a poset is the simplicial complex with one vertex for each element of the poset and a $k$-simplex for each totally ordered chain of $k+1$ elements.    A poset map induces a simplicial map of geometric realizations, and when we use topological terms to describe posets and poset maps we are referring to the geometric realizations.  

\begin{lemma}\label{poset} [Poset lemma \cite{Q}]  Let $f\colon X\to X$ be a poset map, with $f(x)\leq x$ for each $x\in X$ (or $f(x)\geq x$ for each $x\in X$).  Then the image of $X$ is a deformation retract of $X$.  
\end{lemma}

\subsection{Complexes of ideal edges,  ideal trees and the ascending link}
Let $\cI=\cI(\rho)$ be the set of ideal edges of $\rho$.  Recall that a {\em flag complex} is a simplicial complex that is determined by its 1-skeleton: $(k+1)$ vertices span a $k$-simplex if and only if every pair spans an edge.  Let $\bI=\bI(\rho)$ be the flag complex whose vertices are the elements of $\cI$, and whose 1-simplices are pairs $\{\alpha, \beta\}$ of compatible ideal edges.  

Recall that an {\em ideal tree} is a set of pairwise-compatible ideal edges.  Let  $\cT=\cT(\rho)$ be the collection of all ideal trees in $\rho$ and $\bT=\bT(\rho)$ the flag complex whose vertices are the elements of $\cT$, with an edge from $\cA$ to $\cB$ if $\cA\cup \cB\in \cT$.  Then  $\bT$ contains $\bI$ as a subcomplex.

An ideal edge $\alpha$ is {\em ascending for $e_i$} if $\alpha$ splits $e_i$ and $|\alpha|>|e_i|$.
It is {\em ascending} if it is ascending for some $e_i$.  A rose $\rho'$ in the link of $\rho$ in $\bR_n$ is {\em ascending} if $\mu(\rho')>\mu(\rho)$, where $\mu$ is the Morse function $\mu$ defined in section~\ref{Morse}. In this section we describe the homotopy type of the subcomplex $\lk^+(\rho)$ spanned by these ascending roses in terms of ideal edges.
 
For a fixed marked rose $\rho=(R,g)$ let $\cZ(\rho)$ be the set of ascending ideal edges of $\rho$ and $\bZ(\rho)$ the flag complex with vertices $\cZ(\rho)$ and edges compatible pairs $\{\alpha,\beta\}$.    In other words, $\bZ(\rho)$ is the full subcomplex of $\bI(\rho)$ spanned by ascending edges.  

\begin{proposition}\label{prop:ascending} The ascending link $\lk^+(\rho)$ is homotopy equivalent to $\bZ(\rho)$.  
\end{proposition}
\begin{proof}
Every rose in $\lk^+(\rho)$ is obtained from $\rho$ by blowing up some ideal tree $\cA$ and then collapsing a maximal tree $T$ to obtain a new rose, denoted $\rho^\cA_T$. We may assume that $T$ contains none of the blown-up edges, i.e. that $T$ is a subset of the edges $E$   of $\rho$.  Then Lemma~\ref{consequences-of-equality} implies that the sets $\cA\in \cT$ and $T\subset E$ are uniquely determined by $\rho'$, and in particular the map $f\colon \lk^+(\rho)\to \bT(\rho)$ sending $\rho'=\rho^\cA_T$ to $\cA$ is well-defined.  Since $\rho^\cA_T$ and $\rho^\cB_S$ have a common blowup if and only if $\cA\cup\cB$ is an ideal tree,   $f$ is a simplicial map. The inverse image of the simplex spanned by $\cA_0,\ldots,\cA_k$ in $\bT$ consists of roses with the common blowup $\rho^\cA$, where   $\cA= \cA_0 \cup\ldots\cup \cA_k$. Since such roses span a  simplex in $\lk^+(\rho)$, $f$ is a homotopy equivalence onto its image $\im(f)\subseteq \bT$ (see, e.g. \cite{HV}, Cor 2.7).   

Let $\bT^+(\rho)$ denote the subcomplex of $\bT(\rho)$ spanned by forests consisting entirely of ascending edges. 
  If $\cA=\{\alpha_0,\ldots,\alpha_k\}$ is any ideal tree in $\im(f)$, then by the Factorization Lemma (see e.g. \cite{V}, Proposition~7.1), at least one of the $\alpha_i$ must be ascending.  (The Factorization Lemma says that one can match the $\alpha_i$ to edges $e_i\in T$ in such a way that $\alpha_i$ splits $e_i$ for all $i$.)   Inclusion makes $\cT(\rho)$ into a poset, and the map sending $\cA\in \im(f)$ to the subforest $\cA^+$ of ascending edges in $\cA$ is a poset map satisfying the hypotheses of the Poset Lemma. So it is a homotopy equivalence onto its image $\im(f)^+\subset \bT^+(\rho)$.     
  
  Every ideal tree consisting of a single ascending edge is in $\im(f)^+$,  so $\im(f)^+$ contains  $\bZ(\rho)$.  In fact $\bZ(\rho)$ is a deformation retract of $\im(f)^+$.  To see this, define a map $\phi\colon \im(f)^+ \to \bZ(\rho)$ by sending a vertex $\cA$ to the barycenter of the simplex of $\bZ(\rho)$ spanned by the elements of $\cA$, and extending linearly over simplices of $\im(f)^+$.   The deformation retraction is given by sending $x$ to $(1-t)x+t\phi(x)$.  
\end{proof}

\begin{figure}
\begin{center}
\begin{tikzpicture}[scale=1.5] 
\fill [black] (0,0) circle (.05); 
\node [above] (e) at (0,0) {$X$};
\fill [black] (0,1) circle (.05);
\node [above] (f) at (0,1) {$Y$};
\fill [black] (-30:1) circle (.05);
\node [above] (b) at (-30:1) {$W$};
\fill [black] (-150:1) circle (.05);
\node [above] (a) at (-150:1) {$Z$};
\begin{scope} [xshift=.4cm, yshift=-.25cm]
\draw [red,thick, rotate=60] (.1,.0) ellipse (.4cm and .8cm);
\end{scope}
\node   (beta) at (-30:1.5) {$\beta$};
\begin{scope} [xshift=-.4cm, yshift=-.25cm]
\draw [blue, thick, rotate=-60] (-.1,0) ellipse (.4cm and .8cm);
\end{scope}
\node  (alpha) at (-150:1.5) {$\alpha$};
 \end{tikzpicture}
\end{center}
\caption{Key Lemma}
\label{fig:key}
\end{figure}
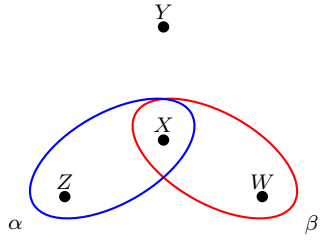

\begin{figure}
\begin{center}
\begin{tikzpicture}[scale=1.75]
  \begin{scope}[xshift=.4cm, yshift=-.25cm]
 \draw (0,0) to (0,1);
    \fill[white] (0,.5) circle (.1);
 \draw (0,0) to (-30:1);
 \draw (0,0) to (-150:1);
        \draw ([shift=(180:1.73cm)]-30:1) arc (180:120:1.73cm);
        \fill[white] (170:.75) circle (.1);
      \draw ([shift=(0:1.73cm)]-150:1) arc (0:60:1.73cm);
      \draw ([shift=(240:1.73cm)]90:1) arc (240:300:1.73cm);
         \fill[white] (-90:.73) circle (.25);
        \fill[white] (-30:.55) circle (.175);
 \fill [black] (-30:1) circle (.05);
  \node [right] (w) at (-30:1) {$W$};
\fill [black] (0,0) circle (.05); 
    \node [below](x) at (0,0) {$X$};
\fill [black] (0,1) circle (.05);    
    \node [above] (y) at (0,1) {$Y$};
\fill [black] (-150:1) circle (.05);
    \node [left] (z) at (-150:1) {$Z$};
    \node [blue,   scale=.75] (a) at (0,.5)    {$X \ndot Y$};
    \node [blue, scale=.75] (c) at (-30:.55)  {$X \ndot W$};
    \node [blue,   scale=.75] (d) at (170:.75)   {$Y \ndot Z$};
    \node [blue,   scale=.75] (f) at (-90:.75)    {$Z \ndot W$};
    \node [blue, below] (g) at (0,-1) {$\text{Contributions to\ }|A|$};
  \end{scope}
  \begin{scope}[xshift=4cm, yshift=-.25cm]
  \draw (0,0) to (0,1);  
    \fill[white] (0,.5) circle (.1);
 \draw (0,0) to (-30:1);  
 \draw (0,0) to (-150:1);  
          \fill[white] (-150:.55) circle (.175);
\draw ([shift=(180:1.73cm)]-30:1) arc (180:120:1.73cm); 
\draw ([shift=(0:1.73cm)]-150:1) arc (0:60:1.73cm); 
         \fill[white] (10:.75) circle (.1);
\draw ([shift=(240:1.73cm)]90:1) arc (240:300:1.73cm);  
         \fill[white] (-90:.73) circle (.25);
 \fill [black] (-30:1) circle (.05);
  \node [right] (w) at (-30:1) {$W$};
\fill [black] (0,0) circle (.05); 
    \node [below](x) at (0,0) {$X$};
\fill [black] (0,1) circle (.05);    
    \node [above] (y) at (0,1) {$Y$};
\fill [black] (-150:1) circle (.05);
    \node [left] (z) at (-150:1) {$Z$};
\node [red,  scale=.75] (a) at (0,.5)    {$Y \ndot X$};
\node [red,  scale=.75] (b) at (-150:.55) {$X \ndot Z$};
    \node [red,  scale=.75] (e) at (10:.75)  {$Y \ndot W$};
    \node [red, scale=.75] (f) at (-90:.73)  {$Z \ndot W$};
    \node [red, below] (g) at (0,-1) {$\text{Contributions to\ }|B|$};
  \end{scope}
\end{tikzpicture}
\end{center}
\caption{Proof of Key Lemma}
\label{fig:keyproof}
\end{figure}
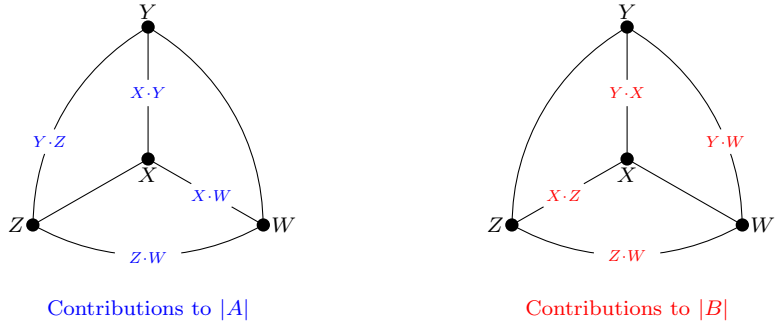

To understand $\bZ(\rho)$, we need to know when ideal edges are ascending. We will do this using the dot product and norm. The following lemma is a  slight variation of \cite[Lemma~11.2]{V}.
\begin{lemma}[Key Lemma] Suppose $E$ is partitioned into four sets $X,Y,Z$ and $W$. Let $A=X\cup Z$ and $B=Y\cup Z$. Then
\[
  |A|+|B| = |X| + |Y| + 2(Z \ndot W)
\]
In particular, if $Z\ndot W>0$ and $|X|,|Y|\geq \lambda \in \R\times\R^{\mathcal W}$, then
$|A|>\lambda$ or $|B|>\lambda$.
\end{lemma}
\begin{proof}
  This is a straightforward computation:
  \begin{align*}
    |A|+|B| &=
    ( X \ndot Y + Y \ndot Z + X \ndot W + Z \ndot W ) +
    ( Y \ndot X + X \ndot Z + Y \ndot W + Z \ndot W )
    \\
    &=
    X \ndot (Y\cup Z \cup W) + Y \ndot (X \cup Z \cup W) + 2 (Z \ndot W)
    \\
    &=
    |X|+|Y| + 2(Z \ndot W)
  \end{align*}
  Figure~\ref{fig:keyproof} may be helpful for following this computation.
\end{proof}

\begin{theorem}\label{thm:main} For $n\geq 2$ and any rose $\rho=(R,g)$ the complex $\bZ(\rho)$ is $(2n-4)$-spherical. 
\end{theorem} 

We will prove the theorem by induction on $n$, but the induction requires considering more general complexes.  
Order the edges of $R$ so that $|e_1|>|e_i|$ for   $i>1$.  Suppose $E$ decomposes into the disjoint union of subsets $X_1,\X_1,\ldots,X_m,\X_m, Y_1,\ldots,Y_k$ with
\begin{itemize}
\item $m\geq 1$
\item $e_i\in X_i$
\item $\overline e_i\in \X_i$
\item $|X_i|,|\X_i|\geq |e_i|$
\end{itemize}

\begin{definition} 
A {\em $V$-ideal edge} is a partition of $$V=\{X_1,\X_1,\ldots,X_m,\X_m, Y_1,\ldots,Y_k\}$$  into two pieces, each with at least two elements, which separates some  $X_i$ from $\X_i$ (recall $m\geq 1$).  
\end{definition}
  Note that a $V$-ideal edge is also an  ideal edge by our old definition, i.e. $V$-ideal edges still split some $e_i$ and can still be classified as  ascending or descending.  Let 
$\bZ(V)$  be the complex whose simplices are sets of pairwise-compatible ascending $V$-ideal edges.  
By the above remark, $\bZ(V)$ is a  subcomplex of $\bZ(\rho).$

\begin{theorem}\label{thm:general}  If $2m+k-4\geq 0$ then the complex $\bZ(V)$ of ascending $V$-ideal edges is $(2m+k-4)$-spherical.   
\end{theorem}

The case $m=n, k=0$ with $ X_i=\{e_i\},\X_i=\{\einv_i\}$  is Theorem~\ref{thm:main}. The proof of the theorem is by induction on the pair $(m,k)$, ordered lexicographically.  We will need the following lemma, which is basically a restatement of the Key Lemma, to get started with $m=1$.
For $m=1$ an ideal edge separates $X_1$ from $\X_1$ so one side is of the form $X_1\cup P$, where $P$ is a proper subset of $\mathbb Y=\{Y_1,\ldots,Y_k\}$.

\begin{lemma}\label{cor:descending}  Suppose $m=1$ and $k\geq 2$, and set $\mathbb Y=\{Y_1,\ldots,Y_k\}$.  Then for  proper subsets $P$ and $Q$   of $\mathbb Y$ we have       
\begin{enumerate}
 \item  If $X_1\cup P$ is descending for $e_1$ then $X_1\cup (\mathbb Y-P)$ is ascending for $e_1$.  
\item  Suppose  $X_1\cup P$ and $X_1\cup Q$  are descending  for $e_1$ and $P\cap Q=\emptyset$.  If $Q\neq \mathbb Y-P$ then $X_1\cup P\cup Q$ is descending  for $e_1$. 
\item  Suppose  $X_1\cup P$ and $X_1\cup Q$  are descending for $e_1$ and $P\cup Q=\mathbb Y$.  If $P\cap Q$ is non-empty, then $X_1\cup (P\cap Q)$  is descending for $e_1$.
\end{enumerate}
\end{lemma}

 \begin{proof} Set $\lambda=|e_1|$.
Then (1) is the Key Lemma with $X=X_1,\, Y=\X_1,$  $Z= P$ and $W=\mathbb Y-P$.
Note that $Z\cdot W>0$ by Lemma~\ref{dot-and-norm} as $P$ and $\mathbb Y - P$ are both non-empty.

For (2), set  $X=X_1,$ $Z=P$, $W=Q$, $Y=(X\cup Z \cup W)^c$, and $\lambda=|e_1|$.  By assumption $|X_1|\geq \lambda$.  If $|X_1\cup P\cup Q|\geq \lambda$ as well, then by the Key Lemma at least one of $|X\cup Z|= |X_1\cup Q|>\lambda $ or $|X\cup W|=|X_1\cup P|>\lambda$, giving a contradiction.  

Statement (3) is proved in the same way with $X=\overline{X_1}$, $Z=\mathbb Y - P$, $W=\mathbb Y -Q$, $Y=(X\cup Z \cup W)^c$, and $\lambda=|\overline{e}_1|=|e_1|$.
\end{proof}

For $m=1$\, Theorem~\ref{thm:general} takes the following stronger form.

\begin{proposition}\label{lem:onek} Let $m=1$ and $k\geq 2$,  If there are no descending $V$-ideal edges then $\bZ(V)$ is    a $(k-2)$-sphere; otherwise it is   a contractible subset of this sphere. 
\end{proposition}

\begin{proof}
We have $V=\{X_1,\X_1,Y_1\ldots, Y_k\}$.     The proof is by induction on $k$. If $k=2$ there are only two ideal edges, represented by  $X_1\cup Y_1$ and $X_1\cup Y_2$.  These are incompatible and at least one of them is ascending, by Lemma~\ref{cor:descending} (1).  
Thus $\bZ(V)$ is either a point or a $0$-sphere.  

For any ideal edge $\alpha$   call the side of $\alpha$ containing $X_1$ the {\em inside} of $\alpha$, and the side containing $\X_1$ the {\em outside}.  The cardinality of the inside is the {\em size} of the ideal edge. For any $k\geq 2$, intersecting the inside of an ideal edge with $\mathbb Y=\{Y_1\ldots, Y_k\}$ gives a one-to-one correspondence between ideal edges and proper subsets of $\mathbb Y$.  In the remainder of the proof we use this correspondence, referring to a proper subset of $\mathbb Y$ as an ideal edge.  
With this convention, the complex of ideal edges can be identified with the barycentric subdivision of the boundary of a $(k-1)$-simplex $\Delta$.  If there are no descending ideal edges, then $\bZ(V)=\bdry\Delta$, which is a $(k-2)$-sphere.  Otherwise, let $D$ be a descending ideal edge of minimal  size.  
Then $A=\mathbb Y-D$  is ascending by Lemma~\ref{cor:descending} (1).

Let $\bZ_A$ be the subcomplex  of $\bZ$ spanned by ascending ideal edges $B$ which do not contain $D$, i.e. such that $B\cup A$ is a proper subset of $\mathbb Y$.  We claim that $B\cup A$ is ascending, so the poset maps $B\to B\cup A\to A$ give a contraction of $\bZ_A$ to the point $A$.  
If $B\subseteq A$ then $B\cup A=A$ is certainly ascending.   If $B$ intersects $D$ note that $B\cap D=(B\cup A)\cap D.$   If $B\cup A$ were descending, then $B\cap D$ would be descending, by Lemma~\ref{cor:descending} (2).  But $B\cap D$ is a proper subset of $D$ and $D$ is minimal, so $B\cap D$ is ascending.

We now add the remaining ascending ideal edges to $\bZ_A$ in order of increasing size, and show that the complex remains contractible after each addition.  Ideal edges of the same size are not compatible, so we may add them independently.  Let 
$\bZ_{\ell}$ denote the subcomplex of $\bZ$ spanned by $\bZ_A$ and all ascending ideal edges of size at most $\ell$, and suppose $B$ has size $\ell+1$, i.e. $B\cap \mathbb Y$ contains $\ell$ elements.  We need to show that the link of $B$ in $\bZ$ intersects $\bZ_{\ell}$ in a contractible subset.  But this intersection consists of ascending subsets of $B$ (ascending sets {\em containing} $B$ are not in $\bZ_{\ell}$ since they contain $D$ and are larger than $B$).  Therefore, replacing $\X_1$ by $\X_1\cup (\mathbb Y-B)=E-B$, we  can identify $\lk(B)\cap \bZ_\ell$ with $\bZ(V')$ for the  partition $V'=\{X_1,E-B,Y_1,\ldots,Y_\ell\}$.  By induction, $\bZ(V')$ is either an $(\ell-2)$-sphere or contractible.  But it is not the whole sphere because it contains a descending ideal edge, namely $D$.  
Therefore the link of $B$ intersects $\bZ_\ell$ in a contractible set.
\end{proof}

We are now in a position to prove Theorem~\ref{thm:general} for all values of $m$ and $k$.

\begin{proof}  We prove the theorem by induction on pairs $(m,k)$, ordered lexicographically.   Lemma~\ref{lem:onek} establishes the theorem for $m=1$ so we may assume $m\geq 2$.   

Let $\bZ_0$ be the subcomplex of $\bZ=\bZ(V)$ spanned by ideal edges which are ascending for some $e_i$ with $i>1$.  Since it is irrelevant whether ideal edges in $\bZ_0$ separate $X_1$ from $\X_1$, we can identify $\bZ_0$ with  $\bZ(V'),$ with $$V'=\{X_2,\X_2,\ldots X_m,\X_m,Y_1,\ldots,Y_k, Y_{k+1}=X_1,Y_{k+2}=\X_1\},$$ which is $(2m+k-4)$-spherical by induction.

Ideal edges in $\bZ-\bZ_0$ split $e_1$ but do not split any $e_i$ for $i>1$, since if they did they would be ascending for that $e_i$ and hence in $\bZ_0$. Therefore the inside of such an ideal edge is of the form $$A=X_1\cup X_2\cup\X_2\cup\ldots\cup X_i\cup \X_i\cup Y_1\cup\ldots\cup Y_j.$$  
  (see Figure~\ref{fig:ideal}). For the remainder of the proof we specify an ideal edge by giving its inside.  
  We will start adding these to $\bZ_0$ in order of increasing size, where the {\em size} of $A$ is the pair $(i,j)$ ordered lexicographically.  Ideal edges of the same size are not compatible, so we may add them independently.
  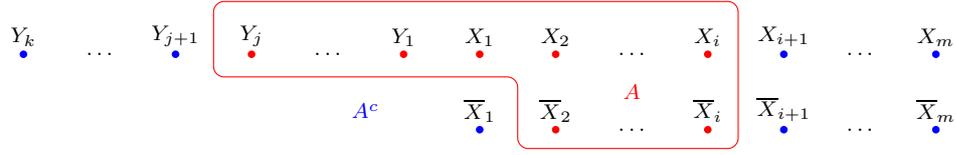
\begin{figure}
\begin{center} 
\begin{tikzpicture}[scale=1] 
\fill [red] (-1,1) circle (.05); 
\node [above](x) at (-1,1) {$ Y_1$};
\node  (x) at (-2,1) {$\ldots$};
\fill [blue] (-4,1) circle (.05); 
\node [above](x) at (-4,1) {$ Y_{j+1}$};
\fill [red] (-3,1) circle (.05); 
\node [above](x) at (-3,1) {$ Y_j$};
\node  (x) at (-5,1) {$\ldots$};
\fill [blue] (-6,1) circle (.05); 
\node [above](x) at (-6,1) {$ Y_k$};
\fill [blue] (0,0) circle (.05); 
\node [above](x) at (0,0) {$\X_1$};
\fill [red] (0,1) circle (.05); 
\node [above](x) at (0,1) {$X_1$};
\fill [red] (1,0) circle (.05); 
\node [above](x) at (1,0) {$\X_2$};
\fill [red] (1,1) circle (.05); 
\node [above](x) at (1,1) {$X_2$};
\node  (x) at (2,0) {$\ldots$}; 
\node  (x) at (2,1) {$\ldots$};
\fill [red] (3,0) circle (.05); 
\node [above](x) at (3,0) {$\X_i$};
\fill [red] (3,1) circle (.05); 
\node [above](x) at (3,1) {$X_i$};
\fill [blue] (4,0) circle (.05); 
\node [above](x) at (4,0) {$\X_{i+1}$};
\fill [blue] (4,1) circle (.05); 
\node [above](x) at (4,1) {$X_{i+1}$}; 
\node  (x) at (5,0) {$\ldots$};
\node  (x) at (5,1) {$\ldots$};
\fill [blue] (6,0) circle (.05); 
\node [above](x) at (6,0) {$\X_m$};
\fill [blue] (6,1) circle (.05); 
\node [above](x) at (6,1) {$X_m$};
\node [red] (A) at (2,.5) {$A$};
\node [blue] (Ac) at (-1.5,.25) {$A^c$};
\draw [rounded corners, red](-3.5,1.7) to (3.4,1.7) to (3.4,-.25) to (.5,-.25) to (.5,.7) to (-3.5,.7) --cycle;
  \end{tikzpicture}
\end{center}
\caption{Ideal edge $\alpha=(A, A^c)$ in $\bZ-\bZ_0,$ of size $(i,j)$}
\label{fig:ideal}
\end{figure}

Let $\alpha\in \bZ-\bZ_0$ with inside $A$ as above, and let $\bZ_{<\alpha}$ denote the complex spanned by $\bZ_0$ and all ideal edges of size less than $(i,j)$.  We may assume that $\bZ_{<\alpha}$ is $(2m+k-4)$-spherical by induction.  To prove that it is still $(2m+k-4)$-spherical after adding $\alpha$ we will show that $\lk(\alpha)\cap   \bZ_{<\alpha}$ is $(2m+k-5)$-spherical.

\begin{figure}
\begin{center}
\begin{tikzpicture}[scale=1] 
\fill [red] (-1,1) circle (.05); 
\node [above](x) at (-1,1) {$Y_1$};
\node  (x) at (-2,1) {$\ldots$};
\fill [red] (-3,1) circle (.05); 
\node [above](x) at (-3,1) {$Y_j$};
\fill [blue] (0,0) circle (.05); 
\node [above, blue](x) at (0,0) {$A^c$};
\fill [red] (0,1) circle (.05); 
\node [above](x) at (0,1) {$X_1$};
\fill [red] (1,0) circle (.05); 
\node [above](x) at (1,0) {$\X_2$};
\fill [red] (1,1) circle (.05); 
\node [above](x) at (1,1) {$X_2$};
\node  (x) at (2,0) {$\ldots$};
\node  (x) at (2,1) {$\ldots$};
\fill [red] (3,0) circle (.05); 
\node [above](x) at (3,0) {$\X_i$};
\fill [red] (3,1) circle (.05); 
\node [above](x) at (3,1) {$X_i$};
  \end{tikzpicture}
\end{center}
\caption{The decomposition $V'$ of $E$}
\label{fig:sublink}
\end{figure}

\begin{figure}
\begin{center}
\begin{tikzpicture}[scale=1] 
\fill [blue] (-1,1) circle (.05); 
\node [above](x) at (-1,1) {$Y_{j+1}$};
\node  (x) at (-2,1) {$\ldots$};
\fill [blue] (-3,1) circle (.05); 
\node [above](x) at (-3,1) {$Y_k$};
\fill [blue] (0,0) circle (.05); 
\node [above](x) at (0,0) {$\X_1$};
\fill [red] (0,1) circle (.05); 
\node [above, red](x) at (0,1) {$A$};
\fill [blue] (1,0) circle (.05); 
\node [above](x) at (1,0) {$\X_{i+1}$};
\fill [blue] (1,1) circle (.05); 
\node [above](x) at (1,1) {$X_{i+1}$};
\node  (x) at (2,0) {$\ldots$};
\node  (x) at (2,1) {$\ldots$};
\fill [blue] (3,0) circle (.05); 
\node [above](x) at (3,0) {$X_m$};
\fill [blue] (3,1) circle (.05); 
\node [above](x) at (3,1) {$X_m$};
  \end{tikzpicture}
\end{center}
\caption{The decomposition $V^{\prime\prime}$ of $E$}
\label{fig:superlink}
\end{figure}

The complex $\lk(\alpha)$ is the join of two subcomplexes, namely the subcomplex spanned by ascending sets contained in  $A$ and that spanned by ascending sets containing $A$.  Ascending sets contained in $A$ are exactly $\bZ(V'),$ where 
$$V'=\{X_1,A^c,X_2,\X_2\ldots,X_i,\X_i, Y_1,\ldots,Y_j\},$$ (see Figure~\ref{fig:sublink}).
Since all ascending sets contained in $A$ are either in $\bZ_0$ or are of smaller size they are all in  $\bZ_{<\alpha},$ i.e. 
$\bZ_{<\alpha}\cap \bZ(V^{\prime})=\bZ(V^{\prime})$, which is $(2i+j-4)$-spherical by induction.
  
Ascending sets containing $A$ are exactly $\bZ(V^{\prime\prime})$, where
 $$V^{\prime\prime}=\{A, \X_1, X_{i+1},\X_{i+1}\ldots,X_m,\X_m,Y_{j+1},\ldots, Y_k\},$$ (see Figure~\ref{fig:superlink}).
If $i<m$ then $\bZ_{<\alpha}\cap \bZ(V^{\prime\prime})$ is isomorphic to $\bZ_0(V^{\prime\prime})$: an ascending $V''$-ideal edge separating only $A$ and $\X_1$ has size larger than $\alpha$ and does not belong to $\bZ_{<\alpha}$; the other ascending $V''$-ideal edges already belong to $\bZ_0(V)$.  Since $\bZ_{<\alpha}\cap \bZ(V^{\prime\prime})\cong\bZ_0(V^{\prime\prime})$ is $(2(m-i+1)+(k-j)-4)$-spherical by induction, the entire complex $\lk(\alpha)\cap \bZ_{<\alpha}$ is $((2i+j-4)) + (2(m-i+1)+(k-j)-4)+1=(2m+k-5)$-spherical, as required.  

This doesn't work if $i=m$ since in that case $V^{\prime\prime}=\{A,\X_1,Y_{j+1},\ldots,Y_k\},$ so $\bZ_0(V^{\prime\prime})$ is empty. The trick is to now   add the remaining edges in order of {\em decreasing} size, as follows: 
 
 \begin{figure}
\begin{center}
\begin{tikzpicture}[scale=1] 
\fill [red] (-1,1) circle (.05); 
\node [above](x) at (-1,1) {$Y_1$};
\node  (x) at (-2,1) {$\ldots$};
\fill [red] (-3,1) circle (.05); 
\node [above](x) at (-3,1) {$Y_j$};
\fill [blue] (0,0) circle (.05); 
\node [above, blue](x) at (0,0) {$A^c$};
\fill [red] (0,1) circle (.05); 
\node [above](x) at (0,1) {$X_1$};
\fill [red] (1,0) circle (.05); 
\node [above](x) at (1,0) {$\X_2$};
\fill [red] (1,1) circle (.05); 
\node [above](x) at (1,1) {$X_2$};
\node  (x) at (2,0) {$\ldots$};
\node  (x) at (2,1) {$\ldots$};
\fill [red] (3,0) circle (.05); 
\node [above](x) at (3,0) {$\X_m$};
\fill [red] (3,1) circle (.05); 
\node [above](x) at (3,1) {$X_m$};
  \end{tikzpicture}
\end{center}
\caption{The decomposition $V'$ of $E$ when $i=m$}
\label{fig:sublink_alt}
\end{figure}

\begin{figure}
\begin{center}
\begin{tikzpicture}[scale=1] 
\fill [blue] (-1,1) circle (.05); 
\node [above](x) at (-1,1) {$Y_{j+1}$};
\node  (x) at (-2,1) {$\ldots$};
\fill [blue] (-3,1) circle (.05); 
\node [above](x) at (-3,1) {$Y_k$};
\fill [blue] (0,0) circle (.05); 
\node [above](x) at (0,0) {$\X_1$};
\fill [red] (0,1) circle (.05); 
\node [above, red](x) at (0,1) {$A$};
  \end{tikzpicture}
\end{center}
\caption{The decomposition $V^{\prime\prime}$ of $E$ when $i=m$}
\label{fig:superlink_alt}
\end{figure}

 Let $\bZ_1(V)$ be the subcomplex of $\bZ(V)$ spanned by $\bZ_0$ and all ideal edges in $\bZ-\bZ_0$ added so far, i.e. of size less than $(m,0)$.  Suppose $\alpha\in \bZ(V)-\bZ_1(V)$.  Then the {\em outside} $A^c$ of $\alpha$ is of the form $\X_1\cup Y_{j+1}\cup\ldots\cup Y_k$.  We will add these to $\bZ_1(V)$ in order of increasing number of $Y's$, checking at each stage   that $\lk(\alpha)\cap   \bZ_{<\alpha}$ is $(2m+k-5)$-spherical. 
  As before, edges of the same size are not compatible so we may consider them separately.  As before, the link of $\alpha$ is the join of $\lk(\alpha)\cap \bZ(V'),$ corresponding to subsets contained in $A$ (whose complement contains $A^c$), and $\lk(\alpha)\cap \bZ(V^{\prime\prime}),$ corresponding to subsets containing $A$ (whose complement is contained in $A^c$).  

{Note that $\bZ_{<\alpha}$ contains everything in $\bZ(V^{\prime})$ except the subsets $A^c\cup P$ for $P\subset\{Y_{1},\ldots,Y_j\}$ as we are adding ideal edges in order of increasing $Y's$-count. Thus, $\bZ_{<\alpha}\cap \bZ(V^\prime)$  is equal to $\bZ_1(V^\prime)$, which is $(2m+j-4)$-spherical by induction.}

{Similarly, we} have already added all subsets $\X_1\cup P$ with $P$ a proper subset of  $\{Y_{j+1},\ldots,Y_k\}$, so $\bZ_{<\alpha}\cap \bZ(V ^{\prime\prime})=\bZ(V ^{\prime\prime}),$ and so  is $(k-j-2)$-spherical.
    Thus $\lk(\alpha)\cap   \bZ_{<\alpha}$ is $(2m+j-4)+(k-j-2)+1=(2m+k-5)$-spherical as required.
\end{proof}

\section{Relation with the Bestvina-Feighn bordification}

In this section we recall Bestvina and Feighn's construction of the bordification $\BFn$, then prove that it is equivariantly homeomorphic to $\BVn$. 

\subsection{Cells of the BF bordification}

Bestvina and Feighn also constructed their bordification one cell at a time, producing an enlargement $\Sigma(G,g)$ of each (open) simplex $\osigma(G,g)$ in $\On$ and showing these $\Sigma(G,g)$  are compatible with face relations. Thus to describe their construction we can drop the marking $g$ from the notation, as we did in Section~\ref{jewels} above.  

They first define an embedding $$i_G\colon\osigma(G) \hookrightarrow \prod_A \csigma(A),$$
where $A$ runs over all core subgraphs of $G$ (including $G$ itself), and $\csigma(A)$ is a closed regular Euclidean simplex with vertices the edges of $A$. 

A point in $\osigma(G)$ is a volume 1 metric on $G$, and the image of this point in the term indexed by $A$ is a volume 1 metric on $A,$  obtained by restricting the metric on $G$ to $A$ and then rescaling.  The space $\Sigma(G,g)$ is then defined to be the closure of the image of $\osigma(G)$ in  $\prod_A \csigma(A)$.  Bestvina and Feighn  prove that $\Sigma(G,g)$ is contractible, and even that it is homeomorphic to a cell; however in the course of proving this last fact they must  appeal to the Poincar\'e conjecture.

Bestvina and Feighn then re-introduce  the markings and show that the spaces $\Sigma(G,g)$  fit together as expected, so that the nerve of the cover is equal to the spine of Outer space.   This shows that the union $\BFn$ of the cells $\Sigma(G,g)$ is contractible.  They then extend the action of $\Outn$ on $\On$ to $\BFn$ and check that the action is proper and cocompact.

\subsection{The faces of $J(G)$} In order to give an explicit description of the faces of $J(G)$ we first establish some notation.

Index the edges of $G$ by the set   $\Delta=\{0,\ldots,m\},$   and let $\scrC$ be the collection of core subsets  of $\Delta,$  corresponding to core subgraphs of $G.$  The {\em core} of an arbitrary subset $U\subset \Delta$ is the maximal element of $\scrC$ contained in $U.$ 
 Note  that the core of $U$ is unique (though it may be empty), since the union of two core subgraphs is a core subgraph. 
 
  Let
  $\mathscr F$ be the set of singletons  $\{i\}$ which are not in $\scrC$ (i.e. the corresponding edges are not loops), and set $\mathscr S= \mathscr F \cup \scrC-\{\Delta\}$.   
  
Recall that the jewel $J(G)$ is defined to be the intersection of the standard simplex $\csigma(G)\subset \R^{m+1}$   with the half spaces $\sum_{i \in A} x_i \geq c_A$ for each core subset $A.$  

Let $e_0,\ldots,e_m$ be the standard basis of $\R^{m+1}$.  
For any subset $S$ of $\Delta$, define the vector $e_S = \sum_{i \in S} e_i$ and let $x_S$ be the function on $\sigma(G)$ defined by $x_S(\vx) = \langle \vx, e_S \rangle=\sum_{i \in S} x_i$, i.e. $x_S$ is the total volume of the subgraph corresponding to $S$ with edge lengths $x_i$. 
If we set $c_S=0$ for $s\in \mathscr F$, then $$J(G)=\{(x_0,\ldots,x_{m})\in  \R^{m+1} \mid   x_0+\ldots+x_m=1 \hbox{ and } x_S\geq c_S \hbox{ for all } S\in\mathscr S\,\}$$

We are now ready to characterize the faces of $J(G)$.

\begin{proposition}\label{faces} Let $G$ be a core graph  and $\J=J(G)$.  With    $\scrC,$ $\mathscr F$  and $\mathscr S$ as above, each codimension $k$ face of $\J$  is given by $k$ equations $x_{S_i}=c_{S_i}$ with $S_i\in \mathscr S$.  Suppose the $S_i$ are ordered so that the size never decreases  and the elements of $\mathscr F$ come first. Let $U_j=\cup_{i=1}^j S_i$. Then for some $t$ between $0$ and $k$ we have 
\begin{enumerate}
\item If $i\leq t$ then $S_i\in\mathscr F$ and 
 the core of $U_i$ is empty. In particular, $U_{t}$ is a forest in $G$.
  \item If $i> t$ then $S_{i}\in\scrC$ and $S_i$ is the core of $U_{i}$. In particular,
    as the $U_{i}$ form an increasing sequence, their cores also increase; i.e.
    \(
    S_{t+1} \subset S_{t+2} \subset \cdots \subset S_{k}
    \).
\item $U_k$ is a proper subset of $\Delta$.
\end{enumerate}
Conversely, any sets $S_1,\ldots,S_k$ satisfying these conditions determine a face of $\J$.
\end{proposition}

\begin{proof}  By Proposition~\ref{prop:rosefaces} the vertices of $J(G)$ lie on rose faces, and there is one vertex for each ordering of the petals.  A rose is obtained by collapsing all of the edges in a maximal tree $T$ of $G$; each of these edges gives a singleton in $\mathscr F$.  The remaining edges are ordered $e_1,\ldots,e_r$, where $r$ is the rank of $H_1(G)$.
What this means for the functions $x_S$ at a vertex of $J(G)$ is that $x_{S_1}=c_{S_1},\ldots,x_{S_m}=c_{S_m}$ where:
\begin{enumerate}
\item $S_1,\ldots,S_{t}$ are singletons and no union of these is in $\scrC$ (these singletons are the edges of   $T$; no core graph is contained in $T$.).
\item  $S_{t+i}$ is the core of the graph spanned by the edges in $T\cup \{e_1,\ldots,e_i\}$.  These form a chain $S_{t+1}\subset\ldots\subset S_{m}$ of core graphs. 
\end{enumerate}
Any subset of a set of equations $x_{S_i}=c_{S_i}$ that determine a vertex, determines a face containing that vertex. Since every face has vertices such subsets in fact give all of the faces.

It remains to show that any collection of sets $S_1\ldots,S_t,S_{t+1},\ldots,S_k$ satisfying conditions (1)-(3) can be enlarged to a full collection of $m$ sets   satisfying the same conditions. Such a full set of hyperplanes will define a vertex of the jewel, which is given by a spanning tree $T$ in $G$ and an ordering of the edges in the complement $G-T$. To ensure compatibility with the given sets $S_1\ldots,S_t,S_{t+1},\ldots,S_k$, we choose the spanning tree $T$ in $G$ so that $S_i$ is the core of $T\cup S_i$ for all $i>t$. One way of finding such a spanning tree runs as follows: first enlarge the forest $U_t\cap S_{t+1}$ to a maximal forest in $S_{t+1}$; add further edges to obtain a maximal forest in $S_{t+2}$, and continue until you obtain a maximal tree $T$ in $G$. This construction ensures that $S_i=\core(T\cup S_i)$ for $i>t$.

Now, choose an ordering of the edges outside $T$ compatible with the chain
  \(
    S_{t+1} \subset S_{t+2} \subset \cdots \subset S_{k}
  \), i.e. first come edges in $S_{t+1}-(T\cap S_{t+1})$ in some order, then the
  edges in $S_{t+2}-((T\cup S_{t+1})\cap S_{t+2})$ in some order, then those from
  $S_{t+3}-((T\cup S_{t+2})\cap S_{t+3})$, etc. A refinement of the chain 
  \(
    S_{t+1} \subset S_{t+2} \subset \cdots \subset S_{k}
  \)
  can be defined by adding one edge at a time, i.e. we find the steps between $U_t$ and $S_{t+1}$
  by adding the edges in $S_{t+1}-(T\cap S_{t+1})$ to $U_{t}$ one at a time in the given order. The steps from $S_{t+1}$ to $S_{t+2}$ are constructed by adding the edges from $S_{t+2}-((T\cup S_{t+1})\cap S_{t+2})$ to $S_{t+1}$ again one at a time in the given order. Continuing this way we obtain a chain of length $m$ that defines a the vertex specified by the rose $G/T$ with an ordering of its petals given by the chosen order of edges outside $T$.
\end{proof}

\begin{remark}\label{exactly-m-hyperplanes-in-any-vertex}
  It follows that any vertex of the jewel $J(G)$ satisfies exactly $m$
  equations $x_{S_i}=c_{S_i}$ with $S_i\in \mathscr S$. The vertex is defined by a
  spanning tree and an ordering of its complementary edges; and the there are only
  $m$ sets in $\mathscr S$ that can occur in a compatible collection: any such
  compatible collection can be extended to a maximal one, which is unique.
\end{remark}

\subsection{The map to $\Sigma(G)$}
Recall that the truncation constants  defining $P=J(G)$ depend only on the rank of the core subgraph, i.e. $c_A=c_r$ where $r=\rank(H_1(A))$.  Recall also that $c_{i} \ll  c_{i+1}\ll 1$ for all $i$.

A map from $\J$ to $\Sigma(G)$ is the same as a family of maps $p_A: P \to \csigma(A)$  for each $A\in\scrC$,  compatible in the sense that the image of $P$ under the product of these maps is contained in $\Sigma(G).$  We identify $\csigma(A)$ with  the positive cone of the projective space $\RP^A=P\R^A$.  We  will choose each map $p_A$ to be a perturbation of the canonical projection from $P$ to $\csigma(A)$  in which we have stretched the map near the boundary to surject onto $\csigma(A)$.  In order to glue our cell-by-cell maps to a map on the whole of $\BVn$, we need to make sure that the various stretches do not interfere with one another. We do this as follows.

For $r\geq 1$ let $g_r\colon[c_r,1]\to[0,1]$ be a smooth  function which is $0$ at $c_r$, strictly increasing between $c_r$ and $c_{r+1}$ and then constantly equal to $1.$  For $S \in \scrC$, set $g_S=g_{r}$ where $r=\rank(H_1(S)).$
 For each $A\in \scrC$
define $\pi_A\colon P\to \R^A$  by 
$$\pi_A(\vx)_i=x_i\prod g_S(x_S),$$
where the product is over all $S\in\scrC$ which contain $i$ but not all of $A$.

\begin{lemma}\label{nonzero}  Let $\vx\in\J$.   For any $A\in \scrC$, some coordinate   $\pi_A(\vx)_i$ is non-zero. 
\end{lemma}

\begin{proof}  To simplify notation, write $z_i=\pi_A(\vx)_i.$

We only need to prove that $\pi_A({\vx})\neq 0$ for vertices ${\vx}$ of $\J$.  This is because vertices satisfy the largest number of relations $x_S=c_S$, making the largest number of $z_i$ equal to zero (since $g_S(c_S)=0$).  Any other point in the boundary satisfies  a subset of these relations. 
  
Since a vertex $\vx $ has codimension $m$ we have $x_{S_i}=c_{S_i}$ for some collection of sets $\{S_1,\ldots,S_m\}\in\mathscr S$ satisfying the conditions of Proposition~\ref{faces}.
We observed in Remark~\ref{exactly-m-hyperplanes-in-any-vertex} that for each set
$S \in \mathscr S - \{S_1,\ldots,S_m\}$, the inequality $x_S \neq c_S$ holds, whence
we have $g_S(x_S) > 0$.

Since the $U_i$ are strictly increasing and $U_m$ is proper, it follows that each $U_i$ must have exactly $i$ elements. 
Without loss of generality we may assume $U_i=\{1,\ldots,i\}$, so that  $0$ is used by none of the $S_i$.

If $A\in\scrC$, we now want to claim some $z_i$ is non-zero, for $i\in A$. 
Let $k$ be the smallest index such that $U_{k}$ contains $A$; if there is no such $k$
then $0\in A$ and $z_0\neq 0$, since only those factors $g_S(x_S)$ contribute to $z_0$
  where $0\in S$, but $0\in S$ implies $S\not\in\{S_1,\ldots,S_m\}$ whence $g_S(x_S)\neq 0$.

Otherwise $k\in A$, and we claim that $z_{k}\neq 0$. The expression for $z_k$ does not use any $S_{i}$ for $i < k$ since those $S_i$ do not contain $k$.    If $i\geq k$ then $U_{i}$ contains $A.$ Since $S_{i}$ is the core of $U_{i}$ and $A$ is a core subset,  $A$ is contained in $S_{i},$ so $g_{S_i}(x_{S_{i}})$ does not occur in the expression for $z_k$.  Since none of the $g_{S_i}(x_{S_{i}})$ occur in the expression for $z_k$, we must have $z_k\neq 0$.  
\end{proof}

By Lemma~\ref{nonzero}, the image of $\pi_A$ misses the origin of $\R^A$, so we can compose it with projection to the projective space $\RP^{A}=P(\R^A)$ to obtain a function  $p_A$. Since   all the terms $g_S(x_S)$ are non-negative on $\J$  the image $p_A (\J)$ is actually in $\RP^A_{\geq 0},$ which is canonically identified with the face $\sigma(A)$ of $\sigma$ spanned by $A$. 

\begin{remark}\label{delta-coordinates}
    The $\pi_A(x)_i$ provide homogeneous coordinates for the point $p_A(x),$ i.e. the point does not change if we multiply all the $\pi_A(x)_i$ by the same nonzero number. In the interior of the jewel $J(G)$, if we use the number $$\prod_{S \supseteq A}g_S(x_S)\neq 0$$ as a multiplier, we find that the point $p_A(x)$ is also given by the homogeneous coordinates
    \(
      \pi_\Delta(x)_i = x_i \prod_{i \in S} g_S(x_S)
    \). In other words, $p_\Delta$ determines $p_A$ in the interior of $J(G)$.
  \end{remark}

The next proposition shows that the maps $p_A$ are compatible with the face relations in $J(G)$ and $\csigma(G)$.  If $\phi$ is a forest in $G$,
let $\kappa_\phi\colon G\to G'$  be the associated forest collapse.  If $\Delta=\{0,\ldots,m\}$ indexes the edges of $G$, then $\phi$ corresponds to a subset $\Phi\subset \Delta$  and $\Delta'=\Delta-\Phi$ indexes the edges of $G'$.  Collapsing a forest sends core graphs to core graphs, so $\kappa_\phi$ induces a map   $\kappa_\phi\colon \scrC\to \scrC'$ sending $A\mapsto A'=A-(A\cap \Phi)$. 

\begin{lemma} The map $\kappa_\phi\colon\scrC\to \scrC'$   has a section  sending $A'$ to the core of $A'\cup\Phi$.  This core has the same rank as $A'$,  and if $B$ is any other core graph in $\kappa_\Phi^{-1}(A')$ then $\rank(B)< \rank(A')$.  
\end{lemma}

\begin{proof}
The core subsets of $\Delta$ that are sent to $A'$ by the collapse are partially
    ordered by inclusion and contain a unique maximal element, namely $A=\core(A'\cup \Phi)$.
  Since $A$ is a core graph, any proper subgraph of $A$ has strictly smaller rank.
\end{proof}

\begin{proposition}\label{prop:commutes} Let $\phi$ be a forest in $G$ and $\kappa_\phi\colon G\to G'$  the associated forest collapse, so that $J(G')$ is a face of $J(G)$ and $\sigma(G')$ is a face of $\sigma(G)$.  Then the following diagram commutes:
\[
\begin{tikzcd}[]
J(G') \arrow[d,"p_{A'}"] \arrow[hookrightarrow, r]& J(G)\arrow[d,"p_{A}"] \\
\csigma(G')\arrow[hookrightarrow, r]& \csigma(G)
\end{tikzcd}
\]
where $A'=\kappa_\phi(A)=A-(A\cap\phi)$.  
\end{proposition} 

\begin{proof}
Since all maps are continuous it suffices to prove the diagram commutes
    for $\vx$ in the interior of $J(G')$. For $\vx=(x_0,\ldots,x_m)$ in the interior of $J(G')$, viewed as a face of $J(G)$, we have that $x_S = c_S$ if and only if $S = \{i\}$ for $i \in \Phi$. Hence the multiplicative factor of Remark~\ref{delta-coordinates} is nonzero for any core subgraph $A$ of $G$, and so $p_\Delta$ determines $p_A$ on $J(G')$. Therefore, it suffices to prove the claim for $A=\Delta$.

In particular, for $\vx=(x_0,\ldots,x_m)$ in $J(G')$, we have $x_i=0$ for $i\in \Phi$. Therefore on $J(G')$,  $x_S=x_{S-(S\cap \Phi)}=x_{S'}$ for all   $S\in\scrC$.  

 We want to show that for each $j\in\Delta-\Phi=\Delta'$ we have $(p_{\Delta'}\vx)_j=(p_\Delta\vx)_j,$ i.e.
 \[
x_j\prod_{\{S'\in\scrC'| j\in S' \}} g_{S'}(x_{S'}) = x_j\prod_{\{S\in\scrC| j\in S\}} g_S(x_S).
\]

The set $\scrC$ breaks into the disjoint union of the sets $\kappa_\phi^{-1}(S')$, for $S'\in \scrC'$. For the maximal set in preimage $S=\core(S'\cup\Phi)$, we have
$r=\rank(S)=\rank(S')$ so $g_S(x_S)=g_r(x_S)=g_r(x_{S'})=g_{S'}(x_{S'})$.  If $T\in\kappa_\phi^{-1}(S')$ is not equal to $S$, then $q=\rank(T)<r$, so $g_{T}(x_{T})=g_{T}(x_{T'})=g_{q}(x_{S'})=1$ since $x_{S'}\geq c_{S'}=c_r $ and $r\geq {q+1}$.  Thus
\[
g_{S'}(x_{S'})=\prod_{T\in \kappa_\phi^{-1}(S')}g_T(x_T)=g_{S}(x_{S}) 
\]
with $S=\core(S'\cup\Phi)$.  Since this is true for all $S'\in\scrC'$, the result follows.
\end{proof}

\subsection{Homeomorphism}   We now  define $p_{\scrC}$ to be the product of all of the $p_A$ defined in the last subsection, i.e. 
$$p_{\scrC}=\prod_{A\in \scrC}p_A\colon \J\to  \prod_{A\in\scrC}  \RP^A_{\geq 0}\iso \prod_{A\in\scrC} \csigma(A).$$  
In this section we show that   $p_{\scrC}$ defines a homeomorphism from $P=J(G)$  to the closure of its image in $\Prod$.  Since $J(G)$ is compact and $\Prod$ is Hausdorff, it suffices to show that $p_{\scrC}$ is injective.

We first show that points on different faces have different images.

Let  $Q$ be a face of $\J$,  determined by subsets $S_1,\ldots,S_t,S_{t+1},\ldots,S_k$ of $\Delta$ satisfying the conditions of Proposition~\ref{faces}. Let $V_0=\Delta-U_k$ and $V_i=U_{k+1-i}-U_{k-i}$ for $i > 0$. Note that the $V_i$ are pairwise disjoint. For notational convenience in what follows, set  $A_0=\Delta$ and $A_i=S_{k+1-i}$ for $1\leq i\leq r=k-t$, so that the $A_i$ are core graphs and $A_0=\Delta\supset A_1=S_k\supset\ldots \supset A_{r}=S_{t+1}$. Since $U_{k+1-i}=U_{k-i}\cup A_i$, we have $V_i\subseteq A_i$.

\begin{lemma}\label{Qint} Let $x$ be a point in the interior of $Q$, let $1\leq\ell\leq r=k-t$ and let $z_i=p_{A_\ell}(x)_i$.  Then $z_i\neq 0$ if and only if $i\in V_\ell$.  
\end{lemma}
 \begin{proof} 
 If $i\in A_\ell-V_\ell$ then either $i\in U_t$ or $i\in A_{\ell+1}\subset A_\ell$ because $A_\ell-V_\ell \subseteq U_{k-\ell} = U_t \cup A_{\ell+1}$. For $i\in U_t$, we have $x_i=0$. For $i\in A_{\ell+1}$, the factor $g_S(x_S)=0$ with $S=A_{\ell+1}$ contributes to $z_i$. In either case, $z_i=0$.
 
Suppose $i\in V_\ell$.  Since $\vx$ is in the interior of $Q$ all  $g_S(x_S)$  are non-zero except when $S = S_j$ for some $j$.     
The expression for $z_i$ does not use any singletons in $U_t$ or any $A_j$ for $j>\ell$ since those $A_j$ do not contain $i$.  If $j\leq \ell$ then $A_j$ contains $A_\ell$ so again is not used in the expression for $z_i$.  Therefore $z_i\ne 0$.  
 \end{proof}

\begin{proposition}\label{local}  If $x$ and $x'$ are in different open faces of $\J$ then $p_{\scrC}(x)\neq p_{\scrC}(x').$
 \end{proposition}
  
  \begin{proof} By Lemma~\ref{Qint} if $x$ is in the interior of $Q$ then $p_\Delta(x)$ determines $V_0=\Delta - U_k$ and therefore determines $U_k$ and $\core(U_k)=S_k=A_1$.  The map $p_{A_1}$ then determines $V_1=U_k-U_{k-1},$ hence $U_{k-1}$ and   $\core(U_{k-1})=S_{k-1}=A_2$. Continuing, we see that the maps $p_{A_i}$   determine all $A_i$  and $U_t$, so determine $Q$.  
 \end{proof}

  We now concentrate on a single face $Q$. 
  
 \begin{proposition}\label{image} Let $\mathscr Q\subseteq \mathscr C$ be the set $\{\Delta,A_1,\ldots,A_r\}$ of core graphs associated to $Q$, and  $p_{\mathscr Q}$ the product map $$p_{\mathscr Q}=p_{\Delta}\times p_{A_{1}}\times\ldots\times p_{A_{r}}\colon Q\to \RP^\Delta_{\geq 0}\times \RP^{A_{1}}_{\geq 0}\times\ldots\times \RP^{A_{r}}_{\geq 0}.$$
Then the image of $Q$ is contained in $\RP^{V_0}_{\geq 0}\times \RP^{V_{1}}_{\geq 0}\times\ldots\times \RP^{V_{r}}_{\geq 0}.$  The boundary of $Q$ maps to the boundary of $\RP^{V_0}_{\geq 0}\times \RP^{V_{1}}_{\geq 0}\times\ldots\times \RP^{V_{r}}_{\geq 0}.$
\end{proposition}

\begin{proof}  It follows immediately from   Lemma~\ref{Qint} that the interior of $Q$ maps to   $\RP^{V_0}_{>0}\times \RP^{V_{1}}_{>0}\times\ldots\times \RP^{V_{r}}_{> 0}$.  Since $p_{\mathscr Q}$ is continuous, all of $Q$ maps to the closure  $\RP^{V_0}_{\geq 0}\times \RP^{V_{1}}_{\geq 0}\times\ldots\times \RP^{V_{r}}_{\geq 0}$.   If $x$ is on the boundary of $Q$ then it is in a different open face $Q'$.  By Proposition~\ref{local} the sets $V'_i$ on which $p_{Q'}$ is non-zero determine this face,  so at least one more coordinate in the sets $V_i$ which determine $Q$ must be zero. 
 \end{proof}
 
 The proof of the following proposition relies on the explicit formula for the function $p_{\mathscr Q}$.

\begin{proposition}\label{localDiffeo}  The map $p_{\mathscr Q}\colon Q\to \RP^{V_0}_{\geq 0}\times \RP^{V_{1}}_{\geq 0}\times\ldots\times \RP^{V_{r}}_{\geq 0}$   is a local diffeomorphism on the interior of $Q$,
\end{proposition}
\begin{proof} 
     Relabeling edges if necessary, we may assume   $0\in V_0$ and $\ell\in V_{\ell}\subseteq A_\ell=S_{k+1-\ell}$  for $\ell=1,\ldots,r=k-t$. 
 
If $i\in V_\ell$ set $z_{i}= p_{A_\ell}(x)_i$, which we may do since $V_\ell\subset A_\ell$. By Lemma~\ref{Qint}, if $x$ is in the interior of $Q$, then all such $z_{i}$ are non-zero; in particular $z_\ell\neq 0$.  Thus  the image of $Q$ is contained in the affine charts $Z_{i}=\frac{z_{i}}{z_\ell}$ on $\RP^{V_\ell}$.

We've defined $z_i=x_i\prod g_S(x_S)$ where the product is over all $S\in\scrC$ which contain $i$ but not all of $A_\ell$. Here, the isolated factor $x_i$ introduces a complication in the formula for the derivative. We can resolve this by defining a new function $\tilde g_S$ for any $S\in\mathscr S$ (as opposed to $S\in \scrC$) by 
  \[
  \tilde g_S(t) =\begin{cases} g_S(t) &\mbox{if $S$ is not a singleton}  \\
                                               t  & \mbox{if } S=\{i\}\in\mathscr F\\
                                               tg_{\{i\}}(t) & \mbox{if $S=\{i\}$ is a loop} 
                                                \end{cases}.
 \]     
We can now rewrite the formula for $z_i$ as
  \[
    z_i = \prod_{\{S\in\mathscr S|i\in S, A_\ell\not\subseteq S\}} \tilde g_S(x_S).
  \]

Let $e_{i\jmath} = e_i - e_j.$  
For $i\in V_\ell$ we now have  $$Z_i=\frac{z_i}{z_\ell}=\frac{\prod_{i\in S, A_\ell\not\subseteq S}{\tilde g}_S(x_S)}{\prod_{\ell\in S, A_\ell\not\subseteq S}{\tilde g}_S(x_S)}=\prod_{S\in\mathscr S} {\tilde g}_S(x_S)^{\langle e_{i\ell},e_S\rangle}$$ 
Note that  $\langle e_{i\ell},e_{S_j}\rangle=0$ for all $j=1,\ldots,k$ since each $S_j$ either contains $A_\ell$ or doesn't contain any element of $V_\ell,$ so it suffices to take the product over all $S\neq S_1,\ldots,S_k$. This is to say, $Z_i$ does not depend on $x_{S}$ for $S\in\{S_1,\ldots,S_k\}$, whence the partial derivative $\frac{\partial Z_i}{\partial x_S}$ vanishes.

The $m-k$  vectors $\{e_{i\ell}\,|\,i\in V_{\ell},\, i\neq \ell\}$  are linearly independent, span an $m-k$ dimensional subspace of $\R^{m+1}$ and are perpendicular to $e_{S_j}$ for all $j=1,\ldots,k$ and to $(1,\ldots,1)$, so give a basis for the tangent space $TQ$ to $Q$.   Let $m_i$ be the coordinates on $TQ$ determined by the vectors $e_{i\ell}$, i.e. $m_i(e_{j\ell})=\delta_{ij}$.  Note that $j$ determines $\ell$ via $j\in V_\ell$. 

We want to show that the derivative $D=Dp_\mathscr Q$ is non-singular.  Since the domain and range both have dimension $m-k$, it suffices to show the kernel of $D$ is zero.  For $v=(v_1,\ldots,v_{m-k})\in TQ$, set $w=(\frac{v_1}{Z_1},\ldots, \frac{v_{m-k}}{Z_{m-k}})$; then to show $\ker(D)=0$ it suffices to show $\langle Dv,w\rangle\neq 0$ for all $v$.  
We have  
\[
  \frac{\partial Z_j}{\partial m_{i}}   =
  \sum_{S\neq S_1,\ldots,S_k}\frac{\partial Z_{j}}{\partial x_S}
                 \frac{\partial x_S}{\partial m_{i}}
\]
 So
\[
  \langle Dv,w\rangle
    =\sum_{i,j}w_j\frac{\partial Z_j}{\partial m_{i}}v_i
    = \sum_{i,j}\sum_{S\neq S_1,\ldots,S_k} w_j
        \left(\frac{\partial Z_j}{\partial x_S}\right)
        \left(\frac{\partial x_S}{\partial m_i}\right)v_i
\]
Since we are only summing over $S$ with $x_S$ nonzero,
\[
  \frac{\partial Z_{j}}{\partial x_S}=\frac{{\tilde g}_S'(x_S)}{{\tilde g}_S(x_S)} \langle e_{j\ell},e_S\rangle Z_j
\]
so 
\begin{align*}
\sum_{i,j}\sum_{S\neq S_1,\ldots,S_k} w_j \left(\frac{\partial Z_j}{\partial x_S}\right)\left(\frac{\partial x_S}{\partial m_i}\right) v_i
&=\sum_{S\neq S_1,\ldots,S_k}\sum_{i,j} w_j  \frac{{\tilde g}_S'(x_S)}{{\tilde g}_S(x_S)}  \langle e_{j\ell},e_S\rangle Z_j   \langle e_{i\ell},e_S\rangle v_i\\
&= \sum_{S\neq S_1,\ldots,S_k}\sum_{i,j} \frac{v_j}{Z_j}  \frac{{\tilde g}_S'(x_S)}{{\tilde g}_S(x_S)} \langle e_{j\ell},e_S\rangle Z_j   \langle e_{i\ell},e_S\rangle v_i\\
&=\sum_{S\neq S_1,\ldots,S_k}\frac{{\tilde g}_S'(x_S)}{{\tilde g}_S(x_S)} (\sum_{i}\langle v_ie_{i\ell},e_S\rangle)^2.
\end{align*}
All summands in the final expression are nonnegative. To conclude that $\langle Dv,w\rangle\neq 0$, we have to argue that
  there is some strictly positive contribution. For singletons, $g'_S$ is strictly positive. Moreover, there must be some singleton $S$ for which $\langle v_ie_{\ell},e_S\rangle$   is not zero, since such $e_S$ are the standard basis for $\R^{m+1}$.  This $S$ cannot be one of the $S_1,\ldots, S_k$ since these are all orthogonal to $TQ$.
\end{proof}

\begin{proposition}\label{Cdiffeo}  $p_{\mathscr Q}$ restricted to the interior of $Q$ 
is a diffeomorphism onto $\RP^{V_0}_{> 0}\times \RP^{V_{1}}_{>0}\times \ldots\times \RP^{V_{r}}_{> 0}.$ 
\end{proposition}

\begin{proof} 
Suppose $K\subset \RP^{V_0}_{> 0}\times \RP^{V_{1}}_{> 0}\times\ldots\times \RP^{V_r}_{> 0}$ is compact. Then $p^{-1}(K)$ is compact in the compact domain $Q$. 
 By Proposition~\ref{image}, $p_{\mathscr Q}$ sends points in $\bdry Q$ to the boundary of $\RP^{V_0}_{\geq 0}\times \RP^{V_{1}}_{\geq 0}\times\ldots\times \RP^{V_r}_{\geq 0},$ so $p_{\mathscr Q}^{-1}(K)$  is actually contained in the interior of $Q$.  Thus $p_{\mathscr Q}$ restricted to the interior of $Q$ is a proper map which is a local diffeomorphism, so it is a covering map.  But $\RP^{V_0}_{> 0}\times \RP^{V_{1}}_{> 0}\times\ldots\times \RP^{V_r}_{> 0}$ is simply-connected, so $p_{\mathscr Q}$ restricted to the interior of $Q$ is a diffeomorphism.   
\end{proof}

Both the domain $\J$ and the codomain $\Prod$ of $p_{\scrC}$ are naturally stratified by the partially ordered set (poset) of their open faces, and the stratification of the codomain induces a stratification of the image.  

\begin{theorem}\label{homeo-on-cells}
  The map $p_{\scrC}$  is a stratum-preserving homeomorphism onto its image which restricts to a diffeomorphism on each stratum.
\end{theorem}

\begin{proof} The map $p_{\scrC}$ is stratified if each $p_A\colon P\to \RP^A_{\geq 0}$ is stratified. The strata of the simplex $\RP^A_{\geq 0}$ are defined by the set of coordinates $z_i$ which vanish on them. An open face $Q$ of $P$ is determined by  the set of coordinates $x_S$ which are non-zero on $Q$. As the $z_i$ are given by monomials in the $x_S$, the set of  $z_i$ which vanish is the same for all $x\in Q$, i.e. $p_A$ maps strata to strata.

Propositions \ref{local} and \ref{Cdiffeo} show that  $p_{\scrC}$ is injective, and the restriction to each stratum is a diffeomorphism onto its image.  Since $p_{\scrC}$ is an injective continuous map from a compact space to a Hausdorff space, it is a homeomorphism onto its image.
\end{proof}

Proposition~\ref{local} could be rephrased to say that the map $p_\scrC$ induces an injection from the poset of strata of $Q$ to the poset of strata of $\Prod$.

As remarked earlier, $\RP^A_{\geq 0}$ is canonically identified with $\sigma(A)$.  The map $p_\Delta$ identifies the interior of $P$ with the interior  $\osigma(G)$ of  $\sigma(G)=\sigma(\Delta)$. If $A$ is a core graph with edges $1,\ldots,m$ and $x$ is in the interior of $P$, we find
$$
  [p_A(x)_1:\ldots:p_A(x)_m]= [p_\Delta(x)_1:\ldots:p_\Delta(x)_m]
$$
because
\(
  p_\Delta(x)_i = x_i\prod_{S \ni i} g_S(x_S) = p_A(x)_i \big( \prod_{S\supseteq A}g_S(x_S)\big)
\)
for any edge $i$ in $A$. Thus we have a commutative diagram
$$
\begin{tikzcd}[]
& \mathring P\arrow[dl, "p_\Delta"'] \arrow[d,"p_{\mathscr C}"]\arrow[hookrightarrow, r]&[-1em] P\arrow[d,"p_{\mathscr C}","\approx"']\\
\osigma(G) \arrow[r, "BF"]& \prod_{A\in\mathscr C}\sigma(A)  &\Sigma(G)\arrow[left hook->, l]
 \end{tikzcd}
 $$ 
where $BF$ is the map defined by Bestvina and Feighn, and the image  $p_{\mathscr C}(P)$  is equal to Bestvina and Feighn's cell $\Sigma(G)$.

\begin{theorem} The jewel space $\BVn$ is  equivariantly homeomorphic to the bordification $\BFn$.
\end{theorem}
\begin{proof}  
Theorem~\ref{homeo-on-cells} yields stratum-preserving homeomorphisms
\[
  p : J(G,g) \rightarrow \Sigma(G,g)
\]
for each marked graph $(g,G)$.  

Now consider a forest collapse $\kappa : G \rightarrow G'$. A core graph
$A$ in $G$ maps to a core graph $A'=\kappa(A)$ in $G'$. The edges of $A'$
can naturally be identified with edges in $A$, whence $\RP^{A'}_{\geq 0}$ is
naturally a face of $\RP^{A}_{\geq 0}$. Although several core graphs in $G$
might give rise to the same core graph $A'$ in $G'$, the diagonal
inclusion
\[
  \prod_{A'\in\mathscr{C}(G')} \RP^{A'}_{\geq 0} \rightarrow
  \prod_{A\in\mathscr{C}(G)} \RP^{A}_{\geq 0}
\]
is well defined and restricts to an inclusion
\(
  \Sigma(G',g') \incl \Sigma(G,g)
\)
where $g$ is a marking on $G$ and $g'$ is the induced marking on $G'$.

The bordification
$\BFn$ is defined by gluing the cells $\Sigma(G,g)$ together,  identifying 
$\Sigma(G',g')$ with its image in $\Sigma(G,g)$ whenever $(G',g')$ can
be obtained from $(G,g)$ by collapsing a forest. In categorical language, the bordification $\BFn$ is the colimit of the functor $(G,g)\mapsto\Sigma(G,g)$,
defined on the category of marked graphs with forest collapses as morphisms.

Similarly, the jewel $J(G',g')$ is a face of the jewel $J(G,g)$ and jewel space
$\BVn$ can be obtained by gluing the jewels along the given identifications, i.e.
$\BVn$ is the colimit of the functor $(G,g)\mapsto J(G,g)$. This follows from the
description given in Section~\ref{fitting} because outer space $\On$ is the colimit
of the cells $\cOsigma(G,g)$ and the deformation retractions
\(
  \cOsigma(G,g) \rightarrow J(G,g)
\)
are compatible with inclusion of faces.

By Proposition~\ref{prop:commutes} the diagrams
$$
\begin{tikzcd}
J(G',g') \arrow[hookrightarrow,r]\arrow[d]& [-.5em]  J(G,g)\arrow[d] \\
  \Sigma(G',g')\arrow[hookrightarrow,r] & \Sigma(G,g)
\end{tikzcd}
$$
are commutative.  Thus  we can paste the homeomorphisms
\(
  J(G,g) \rightarrow \Sigma(G,g)
\)
together to obtain  an equivariant homeomorphism of colimits. 
\end{proof}

\section{The boundary of the bordification}
In a sequel to this paper we will use the identification of  $\BVn$ with $\BFn$ to study the boundary of  the bordification, i.e. the difference between the bordification and its interior. In particular, we show how to cover this boundary by contractible subcomplexes with contractible intersections.  This is analogous to Borel and Serre's covering of the bordification of symmetric space by Euclidean spaces $e(P)$ associated to parabolic subgroups $P$.  In the Borel-Serre case the nerve of the covering is homotopy equivalent to the Tits building of subspaces of a rational vector space, which has the homotopy type of a wedge of spheres.  The top-dimensional homology of the Tits building is the dualizing module.   In our case the  nerve of the covering is homotopy equivalent to a subcomplex of the sphere complex in a doubled handlebody (also called the complex of free splittings), and the relation between its homology and the dualizing module is not so clear.

\affiliationone{Kai-Uwe Bux\\
  Fakult{\"a}t f{\"u}r Mathematik\\
   Universit{\"a}t Bielefeld\\
   Postfach 100131\\
   Universit{\"a}tsstra{\ss}e 25\\
   D-33501 Bielefeld\\
   Germany
   \email{bux@math.uni-bielefeld.de}\\
}
\affiliationtwo{ 
   Peter Smillie\\
   Harvard University\\
Department of Mathematics\\
One Oxford St.\\
Cambridge, MA 02138\\
U.S.A.
\email{smillie@math.harvard.edu}}
\affiliationthree{
Karen Vogtmann\\
Mathematics Institute\\
Zeeman Building\\
University of Warwick\\
Coventry CV4~7AL\\
U.K.
\email{k.vogtmann@warwick.ac.uk} }
\end{document}